\documentclass[10.9pt,a4paper]{amsart}
\usepackage{a4wide}

\usepackage[utf8]{inputenc}
\usepackage[english]{babel}

\usepackage{amsfonts,amssymb,amsmath, enumitem}

\usepackage[nobysame,alphabetic, initials]{amsrefs}

\usepackage{xcolor}
\usepackage[colorlinks,
    linkcolor={red!50!black},
    citecolor={blue!50!black},
    urlcolor={blue!80!black}]{hyperref}

\newcommand{\bsm}{\left(\begin{smallmatrix}}
\newcommand{\esm}{\end{smallmatrix}\right)}
\usepackage{tikz}
\usepackage[all]{xy}
\usetikzlibrary{calc, matrix, arrows, cd}

\newtheorem{innercustomthm}{Theorem}
\newenvironment{customthm}[1]
  {\renewcommand\theinnercustomthm{#1}\innercustomthm}
  {\endinnercustomthm}

\newtheorem{theorem}{Theorem}[section]

\newtheorem{corollary}[theorem]{Corollary}
\newtheorem{lemma}[theorem]{Lemma}

\newtheorem{proposition}[theorem]{Proposition}

\newtheorem{question}[theorem]{Question}

\theoremstyle{definition}
\newtheorem{definition}[theorem]{Definition}

\newtheorem{notation}{Notation}
\newtheorem{example}[theorem]{Example}

\newtheorem{remark}[theorem]{Remark}
\newtheorem{construction}[theorem]{Construction}
\newtheorem{claim}{Claim}
\newtheorem*{claim*}{Claim}

\newcommand{\Z}{\mathbb{Z}}
\newcommand{\Q}{\mathbb{Q}}
\newcommand{\R}{\mathbb{R}}

\newcommand{\id}{\operatorname{id}}

\newcommand{\Fix}{\operatorname{Fix}}

\newcommand{\ks}{\operatorname{ks}}
\newcommand{\fr}{\operatorname{fr}}

\newcommand{\rad}{\operatorname{rad}}

\newcommand{\BTOPPin}{\operatorname{BTOPPin}}

\newcommand{\Pin}{\operatorname{Pin}}

\newcommand{\TOPPin}{\operatorname{TOPPin}}

\newcommand{\BTOP}{\operatorname{BTOP}}

\begin{document}
\title{Involutions on~$S^4$}

\author[K.~Boyle]{Keegan Boyle}
\address{Department of Mathematical Sciences, New Mexico State University, USA}
\email{kboyle@nmsu.edu}

\author[W.~Chen]{Wenzhao Chen}
\address{Institute of Mathematical Sciences, ShanghaiTech University, China}
\email{chenwzh@shanghaitech.edu.cn}

\author[A.~Conway]{Anthony Conway}
\address{The University of Texas at Austin, Austin TX,  USA}
\email{anthony.conway@austin.utexas.edu}

\begin{abstract}
This paper studies locally linear involutions on $S^4$. 
Our main theorem shows that any such involution with a 1–dimensional fixed-point set is necessarily linear, provided the fixed-point set admits an equivariant tubular neighborhood. 
The proof combines modified surgery theory with an equivariant version of the Schoenflies theorem, which we establish here. 
We also show that equivariant tubular neighborhoods of $1$–dimensional fixed-point sets, when they exist, are not unique, in contrast to the nonequivariant case. 
Our results combine with earlier work to provide a classification of all locally linear involutions on $S^4$. 
As a further application, we obtain that strongly negative amphichiral knots with trivial Alexander polynomial are equivariantly topologically slice with respect to the linear action, strengthening a previous result of the first two authors. 
Finally, we also prove that when the fixed-point set is~$2$–dimensional, the involution is linear if and only if the fixed-point set is an unknotted $2$–knot.
\end{abstract}

\maketitle

\section{Introduction}
\label{sec:GeneralStatements}

There has been an extensive study of group actions on 4-manifolds, including the smoothability of locally linear actions, exotic smooth actions, and classifications; see the surveys~\cite{EdmondsTransformation} and~\cite{Chen}.
Progress on the classification problem of locally linear group actions relies on surgery theory. 
For example,  free cyclic group actions on simply-connected 4-manifolds were classified by Hambleton-Kreck and Hambleton-Kreck-Teichner as a consequence of their work on the homeomorphism classification of~$4$-manifolds with cyclic fundamental group~\cite{HambletonKreck,
HambletonKreckCancellation,HambletonKreckTeichnerNonorientable}.
The case of pseudofree locally linear actions, meaning that the action is free away from a discrete set of points with nontrivial stabilizer, was thoroughly studied by Wilczynski and Bauer-Wilczynski~\cite{WilczynskiPeriodic,WilczynskiOnThe,BauerWilczynski}. 


This paper focuses on locally linear involutions on $S^4$.
There are two conjugacy classes of free involutions on $S^4$~\cite{HambletonKreckTeichnerNonorientable}, and
when such actions are not free,  they fix an~$n$-sphere or an integer homology~$3$-sphere.
Work of Kwasik and Schultz shows that up to conjugacy there is a unique locally linear involution on~$S^4$  that fixes two points~\cite[Theorem~2.1]{KwasikSchultz}. When the fixed-point set is 2-dimensional,  Gordon~\cite{Gordon}, building on Giffen's disproof of the generalized Smith conjecture~\cite{Giffen},  produced involutions of $S^4$ which fix knotted 2-spheres.
The classification in the case of 3-dimensional fixed-point sets follows from Freedman's work~\cite{Freedman}.
This article focuses on the remaining 1-dimensional fixed-point set~case, i.e that of a fixed circle. 

\begin{theorem}
\label{thm:1DIntro}
If a locally linear involution $\tau \colon S^4 \to S^4$ fixes a circle and this circle admits an equivariant tubular neighborhood, then $\tau$ is conjugate to the standard linear involution.
\end{theorem}

Here, the \emph{standard linear involution} $S^4 \to S^4$ is obtained by extending the linear involution~$\R^4 \to \R^4,(x,y_1,y_2,y_3)\mapsto (x,-y_1,-y_2,-y_3)$ by fixing the point at infinity. 
Note that when the fixed-point set is~$1$-dimensional, an involution is necessarily orientation-reversing since a locally linear involution of~$S^4$ is orientation-preserving if and only if its fixed-point set is even-dimensional. 
Theorem \ref{thm:1DIntro} is a restatement of Theorem \ref{thm:1dimlclassification} which is proved in Section \ref{sec:1Dim}.

Since fixed-point sets of smooth actions admit equivariant tubular neigbhorhoods (see e.g.~\cite[Chapter VI, Theorem 2.6]{BredonIntroduction}),  Theorem~\ref{thm:1DIntro} admits the following corollary. 

\begin{corollary}
Every smooth involution of $S^4$ that fixes a circle is topologically conjugate to the standard linear action.
\end{corollary}

When the fixed-point set of a locally linear action is of dimension $n \neq 1$, it admits an equivariant tubular neigbhorhood that is unique up to equivariant ambient isotopy; see Proposition~\ref{prop:23dimETN}.
When~$n=1$, it remains unknown whether every involution admits an equivariant tubular neighborhood,  but its existence is a necessary condition for the conclusion of Theorem~\ref{thm:1DIntro} and we prove that if one exists, then it  is not unique.

\begin{theorem} \label{thm:intrononuniquenessIntro}
Assume a locally linear involution on a $4$-manifold $M$ fixes a circle.
If this circle admits an orientable equivariant tubular neighborhood then, up to equivariant homeomorphism of pairs, it admits at least two, and exactly two when $M=S^4$.
\end{theorem}

Theorem~\ref{thm:intrononuniquenessIntro} (which is the combination of Theorem~\ref{thm:nonuniqueness} and Corollary~\ref{cor:other4manifolds}) contrasts with the nonequivariant setting in which Freedman and Quinn proved that a locally flat submanifold of a~$4$-manifold admits a tubular neighborhood that is unique up to ambient isotopy~\cite[Theorem~9.3]{FreedmanQuinn}.

\subsection{An equivariant Schoenflies theorem}
The surprising aspect of the proof of Theorem~\ref{thm:1DIntro} is its reliance on a combination of Kreck's modified surgery~\cite{KreckSurgeryAndDuality} and an equivariant Schoenflies theorem that we establish.
Let $\tau \colon S^4 \to S^4$ be a locally linear involution that fixes a circle.
By hypothesis,  this circle admits an equivariant tubular neigbhorhood $\nu$ and,  in what follows,  we write $X_{\rho,\nu}:=(S^4 - \nu)/\rho$.
We outline the proof of Theorem~\ref{thm:1DIntro} 
\begin{itemize}
\item The first step involves proving that if $\ks(X_{\rho,\nu})=0$, then $\rho$ is conjugate to the standard linear action; here $\ks$ denotes the Kirby-Siebenmann invariant.
This result is proved using Kreck's modified surgery~\cite{KreckSurgeryAndDuality} and more specifically a normal $1$-type argument that relies on a result related to the~$\ell_5$-monoid~$\ell_5(\Z[\Z/2],-)$ from~\cite[Proposition~4]{HambletonKreckTeichnerNonorientable}. 
\item The second step is arguably more suprising as it relies on proving that if $\ks(X_{\rho,\nu})=1$,  then the involution $\rho$ is nonetheless conjugate to the standard linear action.
This unexpected equivariant homeomorphism is not constructed directly, rather, it is obtained by using a doubling trick and applying an equivariant generalization of the Schoenflies theorem. 
\end{itemize}

In order to state our equivariant generalization of Schoenflies theorem we simultaneously use~$\rho_{std}$ to denote a linear involution on $\R^4$, its restriction to~$B^4$, and the involution on~$S^4$ given by compactification.

\begin{theorem}[Equivariant Schoenflies Theorem]
\label{thm:EquivariantSchoenfliesIntro}
Let~$\rho_{std} \colon S^4 \to S^4$ be a standard linear involution and let~$\rho_3 \colon S^3 \to S^3$ be 
a linear involution that satisfies~$
\dim \Fix(\rho_3) = \dim \Fix(\rho_{std})-1$.
If~$f \colon (S^3,\rho_3) \to (S^4, \rho_{std})$ is an equivariant locally flat embedding, then the closure of each component of~$S^4 - f(S^3)$  is equivariantly homeomorphic to~$(B^4, \rho_{std})$.
\end{theorem}

Both assumptions in Theorem~\ref{thm:EquivariantSchoenfliesIntro} are necessary for the conclusion to hold.
Theorem~\ref{thm:EquivariantSchoenfliesIntro} is proved by adapting Brown's proof of the Schoenflies Theorem~\cite{BrownSchoenflies} to the equivariant setting.
Our exposition, which relies on an equivariant Bing shrinking argument,  
is based on Putman's exposition of Brown's result~\cite{Putman}; see Section~\ref{sub:MainProofs}.

\medbreak
We briefly outline the proof of Theorem~\ref{thm:intrononuniquenessIntro}, focusing on the statement that the fixed-point set of the standard linear involution on $S^4$ admits two equivariant tubular neigbhorhoods.
The aforementioned modified surgery argument implies that the fixed-point set of a locally linear involution of $S^4$ admits at most two equivariant tubular neigbhborhoods.
Adapting a surgery theoretic construction of Hambleton-Kreck-Teichner~\cite{HambletonKreckTeichnerNonorientable} produces an involution~$\rho_1$ of $S^4$ whose fixed circle admits a tubular neighborhood $\nu_1$ with $\ks(X_{\rho_1,\nu_1})=1$.
Applying Theorem~\ref{thm:1DIntro},  we deduce that $\rho_1$ is conjugate to the standard linear action and this produces a second equivariant tubular neigbhorhood.

\subsection{An application to equivariant sliceness}

We provide another application of the combination of Theorem~\ref{thm:1DIntro} and the equivariant Schoenflies theorem, this time to a question in equivariant knot concordance.

In \cite{BoyleChen} it is shown that every strongly negative amphichiral knot with trivial Alexander polynomial is equivariantly slice. That is there is some extension~$\rho$ of the point-reflection symmetry on~$S^3$ to~$B^4$ such that there is a~$\rho$-invariant slice disk for~$K$. 
In \cite[Question 1.3]{BoyleChen}, the first two authors ask if there is a non-standard symmetry on~$B^4$ restricting to point-reflection on~$S^3$. 
One potential approach to this question is to consider the involutions constructed in \cite{BoyleChen}. However, Theorem \ref{thm:IntroTest} implies that they are all conjugate to the linear involution. In particular, we can improve \cite[Theorem 1.2]{BoyleChen} to guarantee that the constructed involution is the standard (linear) symmetry on~$B^4$.
In other words, we show that $K$ is \emph{standardly equivariantly} slice, meaning that it is sliced by a disk which is invariant under the linear symmetry with $1$-dimensional fixed-point~set.

\begin{theorem}
If~$K$ is a strongly negative amphichiral knot with~$\Delta_K(t) = 1$, then~$K$ is standardly equivariantly topologically slice. 
\end{theorem}
\begin{proof}
The proof of~\cite[Theorem 1.2]{BoyleChen} constructs an involution $\rho$ on $B_1:= B^4$ that extends the point-reflection involution on $S^3$,  fixes an arc pointwise and fixes a slice disk for $K$ setwise.
The construction in~\cite[Proof of Theorem 1.2]{BoyleChen} ensures that the fixed-point set of~$\rho$ has an equivariant tubular neigbhorhood:
indeed the disk is obtained by first proving that the $0$-surgery on $K$ equivariantly bounds a homotopy $S^1$ and then equivariantly adding a copy of $D^2 \times D^2$.
We can then glue on a $4$-ball~$B_2:= B^4$ with a linear involution that fixes an arc pointwise to get an involution~$\overline{\rho}$ on~$S^4 = B_1 \cup B_2$ whose fixed-point circle admits an equivariant tubular neigborhood.
Apply Theorem \ref{thm:1DIntro} to see that~$\overline{\rho}$ is conjugate to the linear involution. 
Then by the equivariant Schoenflies theorem (Theorem~\ref{thm:EquivariantSchoenfliesIntro}),  the restriction of~$\overline{\rho}$ to~$B_1$ is standard. 
\end{proof}

\subsection{Involutions of the $4$-sphere}

We conclude by describing how Theorem~\ref{thm:1DIntro} complements previous results concerning locally linear involutions on $S^4$.

\begin{theorem}
\label{thm:IntroTest}
Locally linear involutions of the~$4$-sphere are classified as follows.
\begin{enumerate}
\item [(-1)] When the fixed-point set is empty, conjugacy classes of free actions on~$S^4$ correspond to homotopy~$\R P^4$s of which there are exactly two.
\item [(0)] When the fixed-point set is~$0$-dimensional,  it consists of two points.
Up to conjugacy, there is a unique locally linear involution on~$S^4$ that fixes two points.
\item [(1)] When the fixed-point set is~$1$-dimensional,  it is a circle.
There is a unique conjugacy class of locally linear involutions on~$S^4$ whose fixed-point set is one-dimensional and admits an equivariant tubular neigbhorhood.
\item [(2)] When the fixed-point set is~$2$-dimensional,  it is a sphere.
The set of conjugacy classes of locally linear involutions on $S^4$ with $2$-dimensional fixed-point set is in bijection with the set of isotopy classes of~$2$-knots~$S \subset S^4$ whose double branched cover satisfies~$\Sigma_2(S) \cong S^4$.
\item  [(3)] When the fixed-point set is~$3$-dimensional,  it is an integer homology~$3$-sphere.
The set of conjugacy classes of locally linear involutions with~$3$-dimensional fixed-point sets is in  bijection with the set of homeomorphism classes of integer homology~$3$-spheres.
\item [(4)] When the fixed-point set is~$4$-dimensional, the involution is the identity.
\end{enumerate}
\end{theorem}
\begin{proof}

As we recall in Proposition~\ref{prop:FixedPointSetSphere}, the fixed-point set of an involution on~$S^4$, if it is nonempty, consists of an~$n$-sphere or of an integer homology~$3$-sphere.  
The theorem is then a combination of Theorem~\ref{thm:1DIntro} together with Propositions~\ref{prop:Free}, \ref{prop:3D}, \ref{prop:2D} and Theorem~\ref{thm:0dim}.
We emphasise that the result on free actions is due to Hambleton-Kreck-Teichner~\cite{HambletonKreckTeichnerNonorientable} whereas the case where the fixed-point set is~$0$-dimensional follows from the work of Kwasik-Schultz~\cite{KwasikSchultz}.
\end{proof}

\begin{remark}\label{rem:InvolutionsD4}
We also prove a classification of involutions on $B^4$, which is similar to Theorem~\ref{thm:IntroTest}; see Section \ref{sec:D4} for the details.
\end{remark}

\begin{remark}
\label{rem:InvolutionsS3}
By contrast with Theorem~\ref{thm:IntroTest}, it is known that every locally linear involution of~$S^3$ is conjugate to a linear action; see Proposition~\ref{prop:InvolutionsOnS3} for further details.
\end{remark}

Theorem~\ref{thm:IntroTest} shows that when the fixed-point set is not $2$- or $3$-dimensional,  a locally linear action on $S^4$ is conjugate to a linear involution (assuming the existence of an equivariant tubular neighborhood in the $1$-dimensional case). 
In the $3$-dimensional case, the action is conjugate to a linear involution if and only if the fixed-point set is homeomorphic to $S^3$. 
In the~$2$-dimensional case, if the fixed sphere is additionally required to be unknotted,  then the outcome of Theorem~\ref{thm:IntroTest} can be upgraded using results on $4$-manifolds with $\pi_1 \cong \Z$.
In order to state the result,  we refer to the involution $S^4 \to S^4$ obtained by extending the smooth involution~$\R^4 \to \R^4,(x_1,x_2,y_1,y_2)\mapsto~(x_1,x_2,-y_1,-y_2)$ by fixing the point at infinity as the \emph{standard linear involution}. The fixed-point set of this involution is an unknotted~$2$-sphere,  and we show that this property characterises it among locally linear involutions with~$2$-dimensional fixed-point~sets.

\begin{theorem}
\label{thm:S2FixedIntro}
A locally  linear involution of $S^4$ with $2$-dimensional fixed-point set is conjugate to the standard linear action if and only if its fixed-point set is unknotted.
\end{theorem}

\subsection{Open Questions}
\label{sub:Questions}

We conclude with somes open questions.

\begin{question}(Equivariant tubular neigbhorhoods of $1$-dimensional fixed-point sets).
\label{question:HowMany}
Let~$\rho$ be a locally linear involution on an orientable 4-manifold~$M$ whose fixed-point set contains a $1$-dimensional component $F$.
\begin{enumerate}
  \item Does~$F$ have an equivariant tubular neighborhood?
  \item Does~$F$ have at most two equivariant tubular neighborhoods up to equivariant ambient isotopy?
\end{enumerate}
\end{question}

A positive answer to the first question would make it possible to remove the assumption in Theorem~\ref{thm:1DIntro}.
A positive answer to both questions coupled with (an isotopy analogue of) Theorem~\ref{thm:intrononuniquenessIntro} would imply that the fixed-point set of a locally linear involution as in Question~\ref{question:HowMany} admits exactly two equivariant tubular neigbhorhoods up to equivariant ambient isotopy.
Both questions admit a positive answer in the smooth category: indeed recall that the fixed-point set of a smooth action admits an equivariant tubular neigbhorhood that is unique up to equivariant ambient isotopy~\cite[Chapter VI, Theorem 2.6]{BredonIntroduction}.

Following the strategy from~\cite{BrownGluckI,BrownGluckII,BrownGluckIII}, we believe that a positive answer to the first question of Question~\ref{question:HowMany} would follow from a positive answer to the following equivariant analogue of Quinn's annulus theorem~\cite{QuinnEndsIII}, when the fixed-point set is 1-dimensional.

\begin{question}(An equivariant annulus conjecture).
 \label{q:EAC}
Let~$\rho_{std}$ be a linear involution on~$S^4$,  let~$S^3_1,S^3_2 \subset S^4$ be~$3$-spheres that are fixed setwise by~$\rho_{std},$ and let~$A \subset S^4$ be the region cobounded by~$S^3_1$ and $S^3_2.$
Is~$(A, \rho_{std} \mid_{A})$ equivariantly homeomorphic to $(S^3 \times I, \rho_3 \times \id)$, for some linear involution $\rho_3$ on $S^3$?
\end{question}

Quinn's annulus theorem ensures that $A$ is homeomorphic to $S^3 \times I$~\cite{QuinnEndsIII}.
When the fixed-point set of~$\rho_{std}$ is $0$-dimensional,  Question~\ref{q:EAC} admits a positive answer; see Theorem \ref{thm:KwasikSchultz}. 
When the fixed-point set is $3$-dimensional, the answer is also positive, as proved in Theorem~\ref{thm:EquivariantAnnulusThm}.

We conclude with a question in the smooth category.
\begin{question}
For which $n \geq 0$ are there pairs of smooth involutions of $S^4$ with $n$-dimensional fixed-point sets
that are topologically but not smoothly conjugate?
\end{question}

Free actions are excluded because exotic copies of $\R P^4$ exist: Cappell and Shaneson constructed a smooth manifold $R$ that is homotopy equivalent but not diffeomorphic to $\R P^4$~\cite{CappellShaneson} and Hambleton-Kreck-Teichner later proved that $R$ is homeomorphic to $\R P^4$~\cite{HambletonKreckTeichnerNonorientable}. 

\subsection*{Organization}

After reviewing some definitions related to normal bundles and tubular neighborhoods in Section~\ref{sec:Definitions},
we classify locally linear involutions on $S^4$ in increasing order of difficulty.
Section~\ref{sec:-1Dim} recalls how the work of Hambleton-Kreck-Teichner yields the classification of free actions.
Section~\ref{sec:0Dim} shows how the~$0$-dimensional case follows from the work of Kwasik and Schultz.
Section~\ref{sec:3Dim} dispenses with the $3$-dimensional case, whereas Section~\ref{sec:2Dim} focuses on the~$2$-dimensional.
Section~\ref{sec:1Dim} is concerned with the $1$-dimensional case and proves our main results (namely Theorems~\ref{thm:1DIntro},~\ref{thm:EquivariantSchoenfliesIntro} and~\ref{thm:intrononuniquenessIntro}) assuming the surgery theoretic Theorem~\ref{thm:quotientclassification}.
Finally, Section~\ref{sec:Modified} gives an overview of modified surgery theory, and Section~\ref{sec:ProofTechnical} proves Theorem~\ref{thm:quotientclassification}.

\subsection*{Acknowledgments}
AC was partially supported by the NSF grant DMS~2303674.
We are grateful to Mark Powell, Arunima Ray,  and Danny Ruberman for helpful exchanges concerning the existence of tubular neighborhoods.
\subsection*{Conventions}

We work in the topological category with locally flat embeddings and submanifolds as well as locally linear group actions.
Manifolds are assumed to be compact, connected and oriented unless otherwise specified.

\section{Equivariant tubular neighborhoods}
\label{sec:Definitions}

This section introduces some background related to (equivariant) normal bundles and tubular neighborhoods. 
Section~\ref{sub:LocallyLinear} recalls some definitions and facts related to group actions.
Section~\ref{sub:NormalBundle} is concerned with tubular neighborhoods and Section~\ref{sub:Framings} concludes with a discussion of framings.

\subsection{Locally linear group actions}
\label{sub:LocallyLinear}

A group~$G$ acts \emph{locally linearly} on an~$n$-manifold~$M$ if for every~$x \in M$,  there exists a neighborhood~$U$ containing~$x$ that is~$G_x$-invariant and such that~$U$ is~$G_x$-equivariantly homeomorphic to an open set of~$\R^n$ with a linear action; here $G_x$ denotes the stabilizer of $x$.
This definition is designed to ensure that the fixed-point set of a locally linear action is a (locally flat) submanifold.
When~$G=~\Z/2$,  the stabilizers~$G_x$ are either trivial or equal to~$G$, so the definition simplifies: an involution is locally linear if for every fixed point~$x \in M$, there is a~$\Z/2$-invariant neighborhood that is~$\Z/2$-homeomorphic to an open set of~$\R^n$ with a linear involution.

We also record the following application of Smith theory to locally linear involutions of~$S^4$.
\begin{proposition}
\label{prop:FixedPointSetSphere}
The fixed-point set of a locally linear involution on~$S^4$, if it is nonempty, consists of an~$n$-sphere or of an integer homology~$3$-sphere.  
\end{proposition}
\begin{proof}
The local linearity condition ensures that the fixed-point set~$F$ is a submanifold of~$S^4$.
Work of Floyd~\cite[Theorem 4.4]{Floyd} implies that~$b_0^{\Z/2}(F)+b_1^{\Z/2}(F)+\ldots +b_4^{\Z/2}(F) \leq 2$, where we write~$b_i^{\Z_2}(F):=\dim H_i(F;\Z/2)$.
When~$F$ is~$0$-dimensional,  it consists of two points: if the action fixed a single point,  it would restrict to a fixed-point free involution of~$B^4$ which would contradict numerous fixed-point theorems.
When the dimension of~$F$ is at most~$2$, the result follows promptly from the classification of closed~$n$-manifolds, whereas if~$F$ is~$4$-dimensional then a short argument (that we recall in greater generality in Proposition~\ref{prop:23dimETN} below) ensures that~$F=S^4$.

It remains to consider the case where~$F$ is a~$3$-manifold and show that it is a homology~$3$-sphere.
The fixed-point set separates~$S^4$ as~$S^4=V_1 \cup_F V_2$.
The Mayer-Vietoris sequence shows that~$H_2(V_i;\Z/2)=0$ for~$i=1,2$.
\begin{claim}
The involution maps~$V_1$ onto~$V_2$ and vice-versa.
\end{claim}
\begin{proof}[Proof of Claim 1.]
Consider a fixed point~$x$ in the fixed-point set~$F$. 
Since the action is locally linear, we have an order $2$ linear action in a neighborhood of~$x$ with a codimension~$1$ fixed-point set, and hence it is conjugate to a matrix with eigenvalues all~$1$ except a single~$-1$.
Furthermore, since it squares to the identity, it is diagonalizable, so we can take it to be the standard reflection of~$\R^4$ across an~$\R^3$. 
Then we see that locally the two sides of~$F$ are exchanged by the symmetry, from which the claim follows. 
\end{proof}
The claim implies that there is a rel. boundary orientation-reversing homeomorphism~$V_1 \to V_2$.
It follows that~$S^4$ is orientation-preservingly homeomorphic to a~$4$-manifold of the form~$V \cup_F -V$.
The fold map~$r \colon S^4 \to V$ is split by the embedding~$V \to S^4$.
It follows that~$r_* \colon \pi_1(S^4) \to \pi_1(V)$ is a surjection and so~$\pi_1(V)=1$.
Since~$\pi_1(V_i)=1$ and~$H_2(V_i;\Z/2)=0$, it follows that~$H_2(V_i;\Z)=0$ for $i=0,1$.
Thus the~$V_i$ are contractible and so~$F$ is an integer homology sphere.
\end{proof}

\subsection{Equivariant tubular neighbhorhoods}
\label{sub:NormalBundle}

We begin by recalling some definitions in the nonequivariant setting, following mostly the exposition from~\cite[Section 5]{FriedlNagelOrsonPowell}.
Given a closed manifold~$M$, a \emph{closed tubular neigbhorhood} of a closed submanifold~$N^n \subset M^m$ 
consists of a pair~$(\overline{\nu}(N),p \colon \overline{\nu}(N) \to N)$ where~$\overline{\nu}(N)$ is both a neighborhood of~$N$ and a codimension~$0$ submanifold of~$M$, and~$p \colon \overline{\nu}(N) \to N$ is a linear~$D^{m-n}$-bundle such that~$p(x)=x$ for all~$x \in N$.
A \emph{tubular neighborhood} refers to the interior of a closed tubular neighborhood.
Tubular neighborhoods need not exist in general; see e.g.~\cite[Theorem~4]{HirschTubular}. However work of Freedman and Quinn ensures that submanifolds of~$4$-manifolds admit tubular neighborhoods that are unique up to ambient isotopy,  meaning that if~$\overline{\nu}(N)$ and~$\overline{\nu}'(N)$ are tubular neighborhoods, then there is a linear isomorphism~$\Psi \colon \overline{\nu}(N) \to \overline{\nu}'(N)$ such that $\Psi(\overline{\nu}(N))$ and $\overline{\nu}'(N)$ are ambiently isotopic~\cite[Section~9.3]{FreedmanQuinn}; see also the exposition in~\cite[Theorems~5.5 and~5.6]{FriedlNagelOrsonPowell}.

The next definition sets up the equivariant counterparts of these notions.

\begin{definition} \label{def:ETN}
Let~$\rho$ be a locally linear involution on a manifold~$M$, and let~$N$ be a locally flat submanifold which is fixed pointwise by~$\rho$. 
\begin{itemize}
\item A \emph{closed equivariant tubular neighborhood} of~$N$ consists of a closed tubular neighborhood~$(\overline{\nu}(N),p \colon \overline{\nu}(N) \to N)$ together with an order~$2$ linear bundle map~$\rho_\nu \colon \overline{\nu}(N)  \to \overline{\nu}(N)$ that satisfies~$\rho_\nu(x)=x$ for every~$x \in N$.
An \emph{equivariant tubular neighborhood} refers to the interior of a closed equivariant tubular neighborhood.
\item Two closed tubular neighborhoods $\overline{\nu}_0$ and $\overline{\nu}_1$ of $N$ (with inclusion maps $\iota_j \colon \overline{\nu}_i \hookrightarrow N$ for~$j=0,1$)  are \emph{equivariantly ambiently isotopic} if there is an equivariant isomorphism $f \colon \overline{\nu}_0 \to \overline{\nu}_1$ of linear disk bundles such that $\iota_1 \circ f$ and $\iota_0$ are equivariantly ambiently isotopic.
\end{itemize}
\end{definition}

Smooth submanifolds admit equivariant tubular neighborhoods that are unique up to equivariant ambient isotopy~\cite[Chapter VI, Theorem 2.6]{BredonIntroduction}.
We now address
(to the best of our knowledge) what is known about the existence of equivariant tubular neighborhoods in the
topological category in dimension 4. 
To do so, we first state a preliminary theorem, which is a corollary of work of Kwasik-Schultz~\cite[Theorem 2.1]{KwasikSchultz} and Livesay~\cite[Theorem 3]{LivesayFixedPointFree}.

\begin{theorem}
\label{thm:KwasikSchultz}
Any 
free involution~$\rho$ on $S^3 \times I$ that does not swap the boundary components  is conjugate to~$\rho_{std} \times \id_I$, where $\rho_{std}$ denotes the free linear involution on $S^3$ obtained by restricting the action of $-\id_4 \in O(4)$.
\end{theorem}
\begin{proof}
Livesay proved that any free involution on $S^3$ is conjugate to a linear action~\cite[Theorem~3]{LivesayFixedPointFree}.
It follows that~$(S^3 \times \{i\})/\rho$ is homeomorphic to $\R P^3$ for~$i=0,1$.
The orbit set $(S^3 \times I)/\rho$ is then readily seen to be an~$h$-cobordism from $\R P^3$ to itself.
Kwasik and Schultz proved that any~$h$-cobordism from $\R P^3$ to itself is homeomorphic to $\R P^3 \times I$~\cite[Theorem~2.1]{KwasikSchultz}.
Lifting the resulting homeomorphism $(S^3 \times I)/\rho \cong (S^3 \times I)/(\rho_{std} \times \id_I)$ to the universal covers proves the theorem.
\end{proof}
 \color{black}

\begin{proposition}
 \label{prop:23dimETN}
Let~$\rho$ be a locally linear involution on a closed 4-manifold~$M$ with~$n$-dimensional fixed-point set~$F:=\Fix(\rho).$
\begin{itemize}
\item If~$n \neq 1$, then~$F$ admits an equivariant tubular neighborhood.
\item If~$n \neq 1$, then any two equivariant tubular neighborhoods for~$F$ are equivariantly ambiently isotopic.
\end{itemize}
\end{proposition}
\begin{proof}
When~$n=0$, the fixed-point set admits a tubular neighborhood by definition of local linearity.
We argue that when~$n=4$,  the action is trivial and so there is nothing to prove.
Since~$M$ is connected and~$F:=\Fix(\rho)$ is closed, it suffices to prove that~$F$ is also open.
To see this, note that by definition of local linearity, there is a neighborhood~$U$ of~$p \in F$,  with~$(U,\rho|_U) \cong (\R^4,\id)$,  so~$\rho|_U=\id$ and thus~$p \in U \subset F$.

We now assume that~$n \in \{2,3\}$.
We assert that in these cases, the quotient~$M/\rho$ is a manifold,  possibly with boundary.
It is enough to check what happens in a neighborhood of each fixed point, and since the action is locally linear we can find such a neighborhood where the action is given by~$(w,x,y,z) \mapsto (w,x,-y,-z)$ when~$n = 2$ or by~$(w,x,y,z) \mapsto (w,x,y,-z)$ when~$n = 3$. 
In the former case the quotient is again~$\mathbb{R}^4$ and in the latter case the quotient is the upper half-space. 

Now the lift of a tubular neighborhood in the quotient manifold is an equivariant tubular neighborhood in~$M$. 
Additionally, an ambient isotopy in the quotient lifts to an equivariant ambient isotopy in~$M$. 
Since the work of Freedman-Quinn ensures that tubular neighborhoods exist and are unique in the quotient manifold~\cite[Section~9.3]{FreedmanQuinn} (see also the exposition in~\cite[Theorems~5.5 and~5.6]{FriedlNagelOrsonPowell}),  the result follows.

Finally, we consider the uniqueness statement when~$n=0$.
Let~$\overline{\nu}(x)$,~$\overline{\nu}'(x)$ be closed equivariant tubular neighborhoods of~$x  \in M$. 
We begin by applying an equivariant ambient isotopy to~$\overline{\nu}(x)$ which shrinks it linearly until it is contained in the interior of~$\nu'(x)$. 
Apply the $4$-dimensional annulus theorem to see that this collar is homeomorphic to $S^3 \times I$~\cite{QuinnEndsIII}.
Thus restricting $\rho$ to this collar leads to a free action on $S^3 \times I$, and Theorem \ref{thm:KwasikSchultz} implies that this action is conjugate to~$\rho_{std} \times \id$.
Finally, apply an ambient isotopy to~$\overline{\nu}(x)$ which expands it linearly (and hence equivariantly)
 along this standard~$S^3 \times I$ until we have~$\overline{\nu}'(x)$.
\end{proof}

\begin{remark}
\label{rem:ExteriorsMakesSense}
The exterior of the fixed-point set (i.e. the complement of a tubular neighborhood),  considered as a manifold with a group action up to equivariant homeomorphism,  is only a well-defined notion
 when equivariant tubular neigbhorhoods exist and are unique up to equivariant homeomorphism of pairs.
\end{remark}

\subsection{Equivariant framings}
\label{sub:Framings}

Let~$N^n$ be a closed submanifold of a closed manifold~$M^m$.
Following~\cite[Definition~2.2]{KasprowskiPowellRayTeichner},  a \emph{normal bundle} $(\nu(N),\widetilde{\iota})$ of an embedding $\iota \colon N \hookrightarrow M$ is a vector bundle $\nu(N) \to N$ together with an embedding $\widetilde{\iota} \colon \nu(N)  \hookrightarrow M$ that restricts to $\iota$ on the zero section~$s_0$, i.e. such that $\widetilde{\iota} \circ s_0=\iota$.
A \emph{framing} of~$N$ consists of a bundle isomorphism~$\nu(N) \cong N \times \R^{n-k}.$
As is common in low dimensional topology, we will identify the normal bundle of an embedding~$N \hookrightarrow M$ with its image in $M$.
We also frequently omit $\widetilde{\iota}$ from the notation.

\begin{remark}
Embeddings in $4$-manifolds admit normal bundles~\cite[Theorem 9.3]{FreedmanQuinn} but for uniqueness an additional \emph{extendability} condition is needed; we refer to~\cite[Chapter 9]{FreedmanQuinn} for further details as well as to~\cite[Remark 2.5]{KasprowskiPowellRayTeichner} and~\cite[Chapter 6]{FriedlNagelOrsonPowell} for related discussions.
\end{remark}

We now proceed with the corresponding definitions in the equivariant setting.

\begin{definition}
\label{def:EquivariantFraming}
Let~$\rho$ be a locally linear involution on a manifold~$M$, and let~$N$ be a submanifold which is fixed pointwise by~$\rho$. 
\begin{itemize}
\item An \emph{equivariant vector bundle} for $N$ consists of a vector bundle $(\nu(N),\widetilde{\iota})$ for $N$ together with an order $2$ bundle map $\rho_{\nu} \colon \nu(N) \to \nu(N)$ such that $\rho \circ \widetilde{\iota}=\iota \circ \rho_{\nu}$.
\item An \emph{equivariant framing} of $N$ consists of an equivariant vector bundle together with a choice of a bundle isomorphism
$$(\nu(N),\rho_\nu) \cong (N \times \R^{m-n}, \id_N \times -\id_{\R^{m-n}}).$$
\end{itemize}
\end{definition}

For $n\neq 1$, the existence of equivariant normal bundles can be established as in Proposition~\ref{prop:23dimETN} or by taking the interior of a tubular neighborhood.
Taking the disk bundle of an equivariant normal bundle $\nu(N)$ (identified with its image in $M$) yields an equivariant tubular neighborhood.
In particular,  an equivariant framing induces a linear bundle isomorphism
$$(\overline{\nu}(N),\rho_\nu|)  \cong (N \times D^{m-n}, \id_N \times -\id_{\R^{m-n}}|_{D^{m-n}}).$$
In what follows, we will frequently conflate equivariant framings of normal bundles with linear bundle isomorphisms of their disk bundles.
The latter are equivariant tubular neigbhorhoods.

\color{black}

Orientable submanifolds of~$S^4$ admit trivial normal bundles.
This generalises as follows.
\begin{proposition} \label{prop:trivialnormalbundles}
If~$\rho$ is a locally linear involution on~$S^4$ with~$n$-dimensional fixed-point set~$N$,  then any equivariant normal bundle of~$N$ is equivariantly bundle isomorphic to the product bundle~$(N \times \mathbb{R}^{4-n}, \id_N \times -\id_{\R^{4-n}})$.
In particular fixed-point sets of locally linear involutions on~$S^4$ admit equivariant framings, provided they admit an equivariant tubular neighborhood. 
\end{proposition}
\begin{proof}
Proposition~\ref{prop:FixedPointSetSphere} implies that~$N$ is orientable.
Since the normal bundle of any orientable locally flat submanifold of~$S^4$ admits a framing,  the equivariant normal bundle on~$N$, after forgetting the symmetry, is bundle isomorphic to~$N \times \mathbb{R}^{4-n}$. 
Push the action forward to an involution~$N \times \mathbb{R}^{4-n} \to N \times \mathbb{R}^{4-n}$ that we again denote by~$\rho$.
By assumption,~$\rho$ fixes~$N$ pointwise and is therefore a bundle map.
We claim that for any~$x \in N$, the action~$\rho_x$ on the fiber~$\mathbb{R}^{4-n}$ is~$-\id$. 
To see this, note that~$\rho_x^2 = \id$ so that the eigenvalues of~$\rho_x$ are all~$\pm1$. Then after putting~$\rho_x$ in Jordan normal form, we can observe that any off-diagonal entries cause~$\rho_x^2 \neq \id$ so that~$\rho_x$ is diagonalizable. Then since~$N$ was the entire fixed-point set of~$\rho$, the only fixed point of~$\rho_x$ is~$0$ so that~$\rho_x = -\id$.
\end{proof}

\section{Free involutions}
\label{sec:-1Dim}

In this section, we recall how a result of Hambleton-Kreck-Teichner leads to the conjugacy classification of free involutions on $S^4$~\cite{HambletonKreckTeichnerNonorientable}. 
Note that free involutions are automatically locally linear.
\begin{proposition}
\label{prop:Free}
There are two conjugacy classes of free  involutions on~$S^4$, distinguished by the Kirby-Siebenmann invariant of their orbit set.
\end{proposition}
\begin{proof}
We first note a free involution on an even dimensional sphere is orientation reversing by the Lefshetz fixed-point theorem: if it were orientation preserving it would induce multiplication by~$+1$ on top dimensional homology and then there would have to be a fixed point (two in fact).

We assert that mapping an involution to its orbit space defines a bijection between conjugacy classes of free involutions on~$S^4$ and homeomorphism classes of closed nonorientable~$4$-manifolds~$Z$ with~$\pi_1(Z)\cong \Z/2$ and~$\pi_2(Z)=0$.
This map is well defined because,  if~$G$ is finite,  then~$X/G$ is again a manifold and~$X \to X/G$ is a regular~$G$-cover.
The inverse is obtained by mapping a manifold~$Z$ in this set to the deck transformation on its universal cover~$\widetilde{Z}$: this map is well defined because since~$\widetilde{Z}$ is simply-connected and~$\pi_2(\widetilde{Z})=\pi_2(Z)=0$,  Freedman's theorem implies that~$\widetilde{Z}\cong S^4$~\cite{Freedman}.
This proves the assertion.

The proposition now follows from work of Hambleton-Kreck-Teichner~\cite[Theorem 3]{HambletonKreckTeichnerNonorientable} according two which there are only two closed nonorientable~$4$-manifolds~$Z$ with~$\pi_1(Z)\cong \Z/2$ and~$\pi_2(Z)=0$, namely~$\R P^4$ and its fake copy~$\mathcal{R}$, which satisfies~$\ks(\mathcal{R})=1$.
\end{proof}

\section{0-dimensional fixed-point sets}
\label{sec:0Dim}

The aim of this section is to show how Theorem \ref{thm:KwasikSchultz}, a corollary of results of Livesay and Kwasik-Schultz,  implies that that up to conjugacy there is a unique involution of~$S^4$ whose fixed-point set is~$0$-dimensional.
 
 \begin{theorem}
 \label{thm:0dim}
Any two locally linear involutions of~$S^4$ with 0-dimensional fixed-point sets are conjugate.
\end{theorem}
\begin{proof}
Let~$\rho_0,\rho_1 \colon S^4 \to S^4$ be locally linear involutions with fixed-point sets each consisting of two points. 
For each~$\rho_i$, {use Proposition~\ref{prop:23dimETN} to choose equivariant tubular neighborhoods~$\nu_i$ of the two fixed points. 
By Quinn's annulus theorem,~$S^4 - \nu_i$ is homeomorphic to~$S^3 \times I$ (see~\cite{QuinnEndsIII}), and since~$\nu_i$ is~$\rho_i$-invariant,~$\rho_i$ restricts to a free symmetry on~$S^4 - \nu_i$ that does not swap the boundary components. Applying Theorem \ref{thm:KwasikSchultz} (a corollary of \cite[Theorem 2.1]{KwasikSchultz} and~\cite[Theorem~3]{LivesayFixedPointFree}), we have an equivariant homeomorphism $f \colon (S^4 - \nu_0, \rho_0) \to (S^4 - \nu_1, \rho_1)$.

Restrict this homeomorphism to~$\partial S^4 - \nu_0=\partial \overline{\nu}_0$ and observe that any equivariant homeomorphism~$\partial \overline{\nu}_0 \to \partial \overline{\nu}_1$ extends to an equivariant homeomorphism~$\overline{\nu}_0 \to \overline{\nu}_1$ by taking the cone (to see this,  recall that~$(\overline{\nu}_i,\rho_i|) \cong (B^4,\rho_{std})$).}
Thus~$f$ extends to an equivariant homeomorphism~$(S^4, \rho_0) \to (S^4, \rho_1)$. 
\end{proof}

\section{3-dimensional fixed-point sets}
\label{sec:3Dim}

Next, we consider locally linear involutions of $S^4$ with a~$3$-dimensional fixed-point set.
The following proposition is known to the experts (for example it is stated in~\cite[Section~3.1]{EdmondsSurvey}).
We provide a proof here for convenience.
\begin{proposition}
\label{prop:3D}
The set of homeomorphism classes of homology~$3$-spheres is in bijective correspondence with the set of conjugacy classes of locally linear involutions on~$S^4$ with~$3$-dimensional fixed-point set.
\end{proposition}
\begin{proof}
Work of Freedman ensures that any integer homology~$3$-sphere~$Y$ bounds a contractible~$4$-manifold that is unique up to homeomorphism rel.~boundary~\cite[Theorem 1.4$'$]{Freedman}.
Another application of Freedman's work ensures that doubling this contractible manifold gives a~$4$-manifold homeomorphic to~$S^4$.
Interchanging the hemispheres gives rise to an an involution of~$S^4$ with~$Y$ as its fixed-point set.
This defines a map from the set of homeomorphism classes of homology~$3$-spheres to the set of conjugacy classes of locally linear involutions on~$S^4$ with~$3$-dimensional fixed-point set.
For injectivity, note that if~$Y_1$ and~$Y_2$ are integer homology~$3$-spheres for which these involutions are conjugate, then their fixed-point sets are homeomorphic, i.e.~$Y_1$ and~$Y_2$ are homeomorphic.

We prove surjectivity.
The fixed-point set separates~$S^4$ as~$S^4=V_1 \cup_Y V_2$.
Proposition~\ref{prop:FixedPointSetSphere} (and the claim within its proof) shows that the involution maps~$V_1$ homeomorphically onto~$V_2$ (and vice-versa) and that~$Y$ is an integer homology~$3$-sphere.
It follows that~$S^4$ is orientation-preservingly homeomorphic to a~$4$-manifold of the form~$V \cup_Y -V$.
The action is therefore obtained by the doubling process described above, therefore proving surjectivity.
\end{proof}

\section{2-dimensional fixed-point sets}
\label{sec:2Dim}

This section considers locally linear involutions of $S^4$ with $2$-dimensional fixed-point sets and proves Theorem~\ref{thm:S2FixedIntro} from the introduction.
Here and in what follows, we use Proposition~\ref{prop:23dimETN} to guarantee the existence of a unique equivariant tubular neighborhood for~$2$-dimensional fixed-point sets of locally linear involutions; in particular the exterior of the fixed-point set is a well-defined notion in this setting; recall Remark~\ref{rem:ExteriorsMakesSense}.
Proposition~\ref{prop:trivialnormalbundles} ensures that this equivariant tubular neighborhood is equivariantly framable.

\begin{proposition}
\label{prop:2D}
Mapping a~$2$-knot to its~$2$-fold branched cover~$(\Sigma_2(S),\rho)$ defines a bijection from the set of isotopy classes $2$-knots~$S \subset S^4$ with~$\Sigma_2(S)\cong S^4$ to the set of conjugacy classes of locally linear involutions~$\rho$ on~$S^4$ with~$2$-dimensional fixed-point set.
\end{proposition}
\begin{proof}
We begin by proving the injectivity of this assignment.
An orientation-preserving homeomorphism~$\Sigma_2(S_1) \cong \Sigma_2(S_2)$ that intertwines the deck transformations descends to an orientation-preserving homeomorphism~$(S^4,S_1) \cong (S^4,S_2)$.
Since Quinn proved that the mapping class group of~$S^4$ is trivial~\cite{Quinn}, this implies that~$S_1$ and~$S_2$ are isotopic.

For surjectivity, start from a locally linear involution~$\rho \colon S^4 \to S^4$ with~$2$-dimensional fixed-point set~$F$ (recall that necessarily~$F \cong  S^2$) and consider the image of~$F$ in~$X:=S^4/\rho$.
Lemma~\ref{lem:SimplyConnectedOrbifold} below ensures that~$X \cong S^4$ and
it follows that the resulting~$2$-knot in~$X \cong S^4$ has~$2$-fold branched cover homeomorphic to~$S^4$ with~$\rho$ as its deck transformation generator.
\end{proof}

\begin{remark}
Gordon proved that there are infinitely many 2-knots which are the fixed-point set of a locally linear involution on $S^4$~\cite{Gordon}.
\end{remark}

We prove the fact that was stated without proof during the previous argument.

\begin{lemma}
\label{lem:SimplyConnectedOrbifold}
If~$\rho\colon S^4 \to S^4$ is a locally linear involution with a 2-dimensional fixed-point set,  then~$X = S^4/\rho$ is homeomorphic to~$S^4$.
\end{lemma}
\begin{proof}
We claim that~$H_1(X)=0$.
Write~$S^4_F$ for the complement of an equivariant tubular neigbhorhood~$\nu(F)$ of~$F$ and~$X_F$ for the complement of the projection of this neigbhorhood to~$X$.
The projection~$p \colon S^4_F \to X_F$ is a~$2$-fold covering map,  and thus so is~$\partial S^4_F \to \partial X_F.$
Since~$\nu(F)$ is equivariantly framable, both $\partial S^4_F$ and $\partial X_F$ are homeomorphic to $S^1 \times S^2$ and this implies that~$H_1(\partial X_F)\cong \Z$, generated by a meridian of~$p(F) \subset X$.
Since the inclusion induces an isomorphism~$H_1(\partial S^4_F) \xrightarrow{\cong} H_1(S^4_F)$, the naturality of the Serre spectral sequence (together with~$H_2(\Z/2)=0$) yields the following commutative diagram in which the vertical maps are inclusion induced:
$$
\xymatrix@R0.5cm{
0\ar[r]&{\overbrace{H_0(\Z/2;H_1(\partial S^4_F))}^{\in{\{\Z/2,\Z\}}}}\ar[r]\ar[d]^\cong&H_1(\partial X_F)\ar[r]\ar[d]&\Z/2\ar[d]^{\cong}\ar[r]&0\\
0\ar[r]&{\underbrace{H_0(\Z/2;H_1(S^4_F))}_{\in{\{\Z/2,\Z\}}}}\ar[r]&H_1(X_F)\ar[r]&\Z/2\ar[r]&0.\\
}
$$
For~$Y \in \{\partial S^4_F,S^4_F\}$, the leftmost group is~$\Z$ (resp. ~$\Z/2$) if~$\Z/2$ acts trivially (resp.  nontrivially) on~$H_1(Y)$.
Regardless,  this diagram implies that~$H_1(\partial X_F) \to H_1(X_F)$ is an isomorphism.
Thus~$H_1(X_F) \cong \Z$ is generated by a meridian of~$p(F) \subset X$ and attaching~$p(\overline{\nu}(F))$ back kills this meridian. We conclude that~$H_1(X)=0$, as claimed.

Next we argue that~$X$ is simply-connected.
Let~$p\colon S^4 \to X$ be the projection map, and note that~$X$ is naturally an orbifold with orbifold fundamental group~$\pi_1^{orb}(X) \cong \mathbb{Z}/2$ since it is double covered by~$S^4$; see e.g.~\cite[Proposition 2.3.5 item (i)]{caramello}.
Recall that there is a natural surjection~$\pi_1^{orb}(X) \to \pi_1(X)$ (see e.g. \cite[Page 24]{caramello} or \cite{Haefliger}). 
Thus~$\pi_1(X)$ is either trivial or has order $2$.
Since the claim ensures that~$H_1(X) = 0$, we conclude that~$\pi_1(X) = 1$.

The isomorphism~$H_i(X;\Q) \cong H_i(S^4;\Q)^{\Z/2}$ (see e.g.~\cite[Chapter  III, Theorem 7.2]{BredonIntroduction}) ensures that~$b_2(X)=0=b_3(X)$.
Here,  $H_i(S^4;\Q)^{\Z/2}$ denotes the fixed-point set of the action.
Since~$X$ is a closed~$4$-manifold (this was argued during the proof of Proposition~\ref{prop:23dimETN}),  it follows that~$X$ is a homotopy sphere and therefore, by work of Freedman~\cite{Freedman},  it is homeomorphic to a~$4$-sphere.
\end{proof}

Proposition~\ref{prop:2D} shows that a complete description of the set of conjugacy classes of locally linear involutions of~$S^4$ with~$2$-dimensional fixed-point set is challenging to obtain.
The goal of the remainder of this section is to prove that if we require that the complement of the fixed-point set have abelian fundamental group, then only one involution remains.

\begin{lemma}
\label{lem:HomologyS2Fixed}
Let~$\rho \colon S^4 \to S^4$ be a locally linear involution with fixed-point set~$\Fix(\rho)$ an unknotted~$2$-sphere.
The orbit set~$X_\rho:=(S^4 - \nu(\Fix(\rho)))/\rho$ is orientable and satisfies the following properties:
\begin{enumerate}
\item The choice of an equivariant framing~$\nu(\Fix(\rho)) \cong S^2 \times \R^2$\color{black} 
determines a homeomorphism 
$$f_\rho \colon \partial X_\rho \xrightarrow{\cong} S^1 \times S^2.$$
\item The fundamental group of~$X_\rho$ is freely generated by the projection of a meridian to~$\Fix(\rho)$.
\item The Kirby-Siebenmann invariant of~$X_\rho$ is trivial. That is~$\ks(X_\rho)=0$.
\end{enumerate}
\end{lemma}
\begin{proof}
We prove the first assertion.
Pick an equivariant tubular neigbhorhood~$\nu(\Fix(\rho))$ of the fixed-point set.
Upon picking an equivariant framing we obtain an equivariant homeomorphism~$\overline{\nu}(\Fix(\rho)) \to S^2 \times D^2$.
Restricting to the boundary gives an equivariant homeomorphism~$\partial(S^4 - \nu(\Fix(\rho))) \to S^1 \times S^2$.
This descends to a homeomorphism~$f_\rho \colon \partial X_\rho \to S^1 \times S^2$.

We prove the second assertion.
Since the fixed-point set is an unknotted~$2$-sphere,  its exterior satisfies~$\widetilde{X}_\rho\cong D^3 \times S^1.$
The exact sequence of homotopy groups associated to the double covers~$D^3 \times S^1 \cong \widetilde{X}_\rho \to X_\rho$ and~$S^2 \times S^1 \cong \partial \widetilde{X}_\rho \to \partial X_\rho$ gives rise to the following commutative diagram:
$$
\xymatrix{
1 \ar[r]& {\overbrace{\pi_1(\widetilde{X}_\rho)}^{\cong \Z}} \ar[r]& \pi_1(X_\rho) \ar[r]& \Z/2 \ar[r]& 1 \\
1 \ar[r]& {\underbrace{\pi_1(\partial \widetilde{X}_\rho)}_{\cong \Z}} \ar[r]^{\cdot 2}\ar[u]^{i_*}_\cong& {\underbrace{\pi_1(\partial X_\rho)}_{\cong \Z}} \ar[r]\ar[u]^{i_*}& \Z/2 \ar[r]\ar[u]^= & 1.
}
$$
The left hand vertical map is seen to be an isomorphism using the meridian to the unknotted~$2$-sphere. 
It follows that the central map is also an isomorphism.
We deduce that~$\pi_1(X_\rho) \cong \Z$ is freely generated by the projection of a meridian to~$\Fix(\rho)$.

We prove the third assertion.
Since~$\widetilde{X}_\rho \cong S^1 \times D^3$, we have~$\pi_2(X_\rho) \cong \pi_2(\widetilde{X}_\rho)=0$.
Next, since~$H_2(\pi_1(X_\rho))=H_2(\Z)=0$,  the exact sequence~$\pi_2(X_\rho) \to H_2(X_\rho) \to H_2(\pi_1(X_\rho)) \to 0$ ensures that~$H_2(X_\rho)=0$.
We deduce that~$X_\rho$ has trivial intersection form.
It follows both that~$X_\rho$ is spin and has vanishing signature.
We conclude that~$X_\rho$ is spin and thus~$\ks(X_\rho)=\sigma(X_\rho)=0$.
\end{proof}

The proof of Lemma~\ref{lem:HomologyS2Fixed}
shows that $X_\rho$ is a homotopy $S^1 \times D^3$ with $\pi_1(X_\rho) \cong \Z$ and for which the inclusion induced map~$\pi_1(\partial X_\rho) \to \pi_1(X_\rho)$ is an isomorphism.
Arguments in surgery theory then readily show that the structure set $\mathcal{S}(X_\rho,\partial X_\rho)$ is trivial from which it follows that~$\rho$ is conjugate to a linear action.
In what follows, we provide a more down to earth (but nevertheless fairly short) proof of this result. 

We wish to apply~\cite[Theorem 1.10]{ConwayPowell} in order to prove that any two orbit sets~$X_{\rho_0}$ and~$X_{\rho_1}$ as above are homeomorphic.
For this, we need an intermediate construction.

\begin{construction}[A map~$\pi_1(\partial X_\rho) \to \Z$]
\label{cons:MapToZ}
Let~$\rho \colon S^4 \to S^4$ be an orientation-preserving involution with fixed-point set~$\Fix(\rho)$ an unknotted~$2$-sphere.
Lemma~\ref{lem:HomologyS2Fixed} ensures that the choice of the tubular neighborhood~$\nu(\Fix(\rho))$ of~$\Fix(\rho)$ and of an equivariant framing~$\overline{\nu}(\Fix(\rho)) \cong S^2 \times D^2$ determine a homeomorphism~$f_\rho \colon \partial X_\rho \xrightarrow{\cong} S^1 \times S^2$.
Consider the isomorphism
$$\varphi \colon  \pi_1(\partial X_\rho) \xrightarrow{f_\rho} \pi_1(S^1 \times S^2) \cong \Z.$$  
Lemma~\ref{lem:HomologyS2Fixed} shows that the choice of an oriented meridian for~$\Fix(\rho)$ determines an isomorphism~$\pi_1(X_\rho) \cong \Z$.
Observe that~$\varphi$ agrees with the composition~$\pi_1(\partial X_\rho) \xrightarrow{i_*,\cong} \pi_1(X_\rho) \cong \Z$,  as it sends the projection of the meridian to~$1$.
Note also that the corresponding Alexander module is trivial:~$H_1(\partial X_\rho;\Z[\Z])=0$.
\end{construction}

Recall that the \emph{standard linear involution} on $S^4$ with~$2$-dimensional fixed-point set is obtained by extending the linear involution~$\R^4 \to \R^4,(x_1,x_2,y_1,y_2)\mapsto (x_1,x_2,-y_1,-y_2)$ by fixing the point at infinity.
The next result proves Theorem~\ref{thm:S2FixedIntro} from the introduction.

\begin{theorem}
\label{thm:S2Fixed}
A locally  linear involution of $S^4$ with $2$-dimensional fixed-point set is conjugate to the standard linear action if and only if its fixed-point set is unknotted.
\end{theorem}
\begin{proof}
The standard linear action fixes an unknotted $2$-sphere, so we focus on the converse.
We will show that if~$\rho_0,\rho_1 \colon S^4 \to S^4$ are orientation preserving involutions with fixed-point sets~$\Fix(\rho_i)$ each an unknotted~$2$-sphere, then~$\rho_0$ and~$\rho_1$ are conjugate.
\color{black}
For~$i=0,1$, let~$\nu(\Fix(\rho_i))$ be equivariant tubular neighorhoods of the fixed-point set and set~$X_i:=X_{\rho_i}$ for brevity.
Upon picking an equivariant framing we obtain an equivariant linear bundle isomorphism~$\overline{\nu}(\Fix(\rho_i)) \to S^2 \times D^2$.
Restricting to the boundary gives an equivariant homeomorphism~$\partial(S^4 - \nu(\Fix(\rho_i))) \to S^2 \times S^1$.
This action descends to a homeomorphism~$f_i \colon \partial X_i \to S^1 \times S^2$.

Consider~$f:= f_1 \circ f_0^{-1} \colon \partial X_0 \to \partial X_1$ and the coefficient systems~$\varphi_i \colon \pi_1(\partial X_i) \to \Z$ from Construction~\ref{cons:MapToZ} for $i=0,1$.
We assert that~$f_* \circ \varphi_0=\varphi_1$.
It suffices to argue that~$f$ maps one projected meridian to the other.
To see this,  note that the framings take meridians to meridians (because they are linear bundle isomorphisms).
Since the framings are equivariant, the same is true on the orbit sets.

Since the $\Z$-covers
of the~$X_i$ are homeomorphic to~$\R \times D^3$,
their~$\Z[\Z]$-intersection form and the Alexander module of their boundary are trivial.
The third item of Lemma~\ref{lem:HomologyS2Fixed} ensures that~$\ks(X_0)=0=\ks(X_1)$.
It now follows from~\cite[Theorem 1.10]{ConwayPowell} that~$f$ extends to a homeomorphism~$X_0 \to X_1$.

Lift this homeomorphism to the universal cover to obtain an equivariant homeomorphism 
$$(S^4 - \nu(\Fix(\rho_0)) \to (S^4 - \nu(\Fix(\rho_1)).$$
By construction, this homeomorphism extends over the equivariant tubular neigbhorhoods.
Combining these equivariant homeomorphisms gives the required homeomorphism~$S^4 \to S^4$.
\end{proof}

\section{1-dimensional fixed-point sets}
\label{sec:1Dim}

This section focuses on involutions of~$S^4$ whose fixed-point set is~$1$-dimensional and proves our main results, namely Theorems~\ref{thm:1DIntro},~\ref{thm:EquivariantSchoenfliesIntro} and~\ref{thm:intrononuniquenessIntro} from the introduction.
Throughout the section, several constructions will require an equivariant tubular neighborhood.
As noted in Section~\ref{sec:Definitions},  it is not known whether every locally linear involution has an equivariant tubular neighborhood.
If the involution is smooth, then the existence of an equivariant tubular is well-known; see for example~\cite[Chapter VI, Theorem 2.2]{BredonIntroduction}. 
Since we have been unable to show the corresponding result in the topological category,  we make the following definition.

\begin{definition}
\label{def:FixedPointLinear}
A locally linear involution is \emph{fixed-point linear} if its fixed-point set admits an equivariant tubular neigbhorhood.
\end{definition}

Recall from Proposition~\ref{prop:23dimETN} shows that when $n\neq 1$, every locally linear involution on a closed~$4$-manifold with $n$-dimensional fixed-point set is fixed-point linear.
When the fixed-point set is~$1$-dimensional, an example of a fixed-point linear involution on $S^4$ arises from the smooth involution~$\R^4 \to \R^4,(x_1,y_1,y_2,y_3) \mapsto (x_1,-y_1,-y_2,-y_3)$ with fixed-point set homeomorphic to~$\mathbb{R}$.
Indeed this involution extends to an involution~$\rho_0 \colon S^4 \to S^4$ by fixing the point at infinity.
The goal of this section is to prove the following results which are Theorem~\ref{thm:1DIntro} and (most of) Theorem~\ref{thm:intrononuniquenessIntro} from the introduction.

\begin{theorem}\label{thm:1dimlclassification}
There is a unique fixed-point linear involution~$S^4 \to S^4$ with fixed-point set homeomorphic to a circle, up to conjugation in the homeomorphism group of~$S^4$; this involution is~$\rho_0$.
\end{theorem}
\begin{theorem} \label{thm:nonuniqueness}
The fixed-point set of the involution~$\rho$ in Theorem \ref{thm:1dimlclassification} has exactly two equivariant tubular neighborhoods~$\nu_0,\nu_1$ up to equivariant homeomorphism of 
pairs~$(S^4, \nu)$. 
The neigbhorhoods satisfy~$\ks((S^4 - \nu_0)/\rho) = 0$ and~$\ks((S^4 - \nu_1)/\rho) = 1$.
\end{theorem}

Before proceeding, we observe in the following corollary that the non-uniqueness in Theorem~\ref{thm:nonuniqueness} can be extended to other 4-manifolds, thus proving the remainder of Theorem~\ref{thm:intrononuniquenessIntro} from the introduction.
For the following corollary,  let~$\rho_{std} \times \id_{S^1}$ be the free involution on~$S^2 \times S^1$ given by~$\rho_{std}(x,y) = (-x,y)$ and slightly abusing notation let~$\rho_{std}$ also refer to the involution on~$D^3 \times S^1$ given by~$\rho_{std}(x,y) = (-x,y)$.

\begin{corollary} \label{cor:other4manifolds}
If~$\rho$ is a locally linear involution on a 4-manifold~$M$ with a 1-dimensional component~$F$ of~$\Fix(\rho)$, and~$F$ has an orientable equivariant tubular neighborhood~$\nu(F)$, then~$F$ has at least two equivariant tubular neighborhoods up to equivariant homeomorphism of pairs. 
\end{corollary}
\begin{proof}
We will assume that~$F$ is the only component of~$\Fix(\rho)$. Indeed, if there is any other component of~$\Fix(\rho)$,  remove a neighborhood of it and consider the resulting manifold with boundary. 

We begin by considering a pair of equivariant tubular neighborhoods~$\overline{\nu}_0$ and~$\overline{\nu}_1$ of the fixed~$S^1$ in~$S^4$ with respect to the standard involution~$\rho_{std}$ as guaranteed by Theorem \ref{thm:nonuniqueness}. 
Applying an equivariant ambient isotopy shrinking~$\overline{\nu}_1$ fiber-wise until~$\overline{\nu}_1 \subset \nu_0$, we can consider the remaining shell~$(W,\rho) := (\overline{\nu}_0 - \nu_1, \rho_{std} \mid_{\overline{\nu}_0 - \nu_1})$. 
Since~$\ks((S^4 - \nu_0)/\rho) = 0$ and~$\ks((S^4 - \nu_1)/\rho) = 1$,  the additivity of the Kirby-Siebenmann invariant (see e.g.~\cite[Theorem 9.2]{FriedlNagelOrsonPowell}) implies that~$\ks(W/\rho) = 1$.

Now let~$\nu_0(F)$ be an equivariant tubular neighborhood of~$F$ in~$M$. 
Since~$\nu_0(F)$ is orientable the proof of Proposition~\ref{prop:trivialnormalbundles} implies that $F$ has an equivariant framing which gives an equivariant bundle isomorphism~$\psi \colon \nu_0(F) \to \nu_0$, where~$\nu_0$ is the neighborhood of~$S^1$ in~$S^4$ used above. 
We can then consider~$\nu_1(F):=\psi^{-1}(\nu_1)$ which is another (orientable) equivariant tubular neighborhood of~$F$ in~$M$. 
Note that $\nu_1(F)$ equivariantly embeds in $M$, since $\nu_1(F) \subset \nu_0(F)$.
We now note that~$\ks((M - \nu_0(F))/\rho) \neq \ks((M - \nu_1(F))/\rho)$, since their difference is~$\psi^{-1}(W)/\rho$ which has Kirby-Siebenmann invariant~$1$. 
We conclude that~$\nu_1(F)$ and~$\nu_0(F)$ do not have equivariantly homeomorphic exteriors. 
\end{proof}
\subsection{A technical intermediate result}

Theorems \ref{thm:1dimlclassification} and \ref{thm:nonuniqueness} (which are proved in Section~\ref{thm:nonuniqueness} below) will follow from a more technical theorem which we set up with the following notation.
Given 
a fixed-point linear involution~$\rho \colon S^4 \to S^4$ with $1$-dimensional fixed-point set and an equivariant tubular neighorhood~$\nu$ of $\Fix(\rho)$,  set
$$X_{\rho,\nu}:=(S^4 - \nu)/\rho.$$
The main technical result needed to prove Theorems \ref{thm:1dimlclassification} and \ref{thm:nonuniqueness}  is the following.
\begin{theorem} \label{thm:quotientclassification}
Let~$\rho,\rho' \colon S^4 \to S^4$ be a fixed-point linear involutions with 1-dimensional fixed-point sets, and let~$\nu,\nu'$ be equivariant tubular neighborhoods of~$\Fix(\rho)$ and~$\Fix(\rho')$ respectively. 
The following are equivalent:
\begin{enumerate}
  \item the Kirby-Siebenmann invariants of the orbit sets satisfy~$\ks(X_{\rho,\nu}) = \ks(X_{\rho',\nu'})$, 
  \item there is a homeomorphism~$X_{\rho,\nu} \cong X_{\rho',\nu'}$.
\end{enumerate}
Additionally:
\begin{itemize}
\item Both values of the Kirby-Siebenmann invariant are realized: there are~$(\rho_0,\nu_0)$ and~$(\rho_1,\nu_1)$ with~$(S^4 - \nu_0,\rho_0) \cong (S^2\times D^2,\rho_{std})$ so that~$\ks(X_{\rho_0,\nu_0})=0$, and~$
S^4 - \nu_1 \simeq S^2 \times D^2$ with~$\ks(X_{\rho_1,\nu_1})=1$.
\item If $\ks(X_{\rho,\nu}) = \ks(X_{\rho',\nu'})$,  then  there is an equivariant homeomorphism $S^4 - \nu \to S^4 - \nu'$ that is either rel. boundary or extends the Gluck twist (see Remark \ref{rem:quotientclassification} below); in particular~$\rho$ and~$\rho'$ are conjugate. 
\end{itemize} 
\end{theorem}

The involutions $\rho_0$ and $\rho_1$ are described in Constructions~\ref{cons:1dimKS=0} and~\ref{cons:1dimKS=1} below.

\begin{remark}
\label{rem:quotientclassification}
We comment on the last statement of Theorem~\ref{thm:quotientclassification}
\begin{itemize}
\item 
Given an equivariant tubular neighborhood~$\nu$ of the fixed-point set of an involution~$\rho$ as in Therorem~\ref{thm:quotientclassification}, an equivariant framing induces an equivariant linear bundle isomorphism~$(\overline{\nu},\rho|) \xrightarrow{\cong} (S^1\times D^3, \rho_{std})$ and therefore an equivariant homeomorphism $\partial (X- \nu) \cong (S^2 \times S^1)$ and a homeomorphism~$\partial X_{\rho,\nu} \cong (\R P^2 \times S^1)$.
The last statement in Theorem~\ref{thm:quotientclassification} asserts that given equivariant framings~$\fr \colon \overline{\nu} \to S^1 \times D^3$ and~$\fr' \colon \overline{\nu}' \to S^1 \times D^3$,  if~$\ks(X_{\rho,\nu}) = \ks(X_{\rho',\nu'})$ then there is an equivariant homeomorphism $S^4 - \nu \to S^4 - \nu'$ extending either $\fr^{-1}\vert_{\partial}\circ \fr'\vert_{\partial}$ or $\fr^{-1}\vert_{\partial}\circ G \circ \fr'\vert_{\partial}$, where $G$ is the Gluck twist. 
\item 
To see that $\rho$ and $\rho'$ are conjugate, use that equivariant framings and the Gluck twist extend equivariantly over the respective copies of $S^1 \times D^3$.
\end{itemize}
\end{remark}

The proof of Theorem~\ref{thm:quotientclassification}  uses modified surgery theory and will be given in Section~\ref{sec:ProofTechnical} after we introduce the necessary surgery theoretic background in Section~\ref{sec:Modified}.
In the course of the proof we will see that~$X_{\rho_0,\nu_0} \cong \R P^2 \times D^2$ and~$X_{\rho_1,\nu_1}$ is a fake~$\R P^2 \times D^2$. We will also observe that the actions~$\rho_0$ and~$\rho_1$ are in fact conjugate in the homeomorphism group of~$S^4$. In particular, it is the choice of equivariant tubular neighborhood, not the involution itself, which distinguishes~$X_{\rho_0,\nu_0}$ and~$X_{\rho_1,\nu_1}$.
Finally, we note that~$\rho_0$ is the standard involution of~$S^4$ constructed above Theorem~\ref{thm:1dimlclassification}.

\subsection{Proof of Theorems \ref{thm:1dimlclassification} and \ref{thm:nonuniqueness}}
\label{sub:MainProofs}

To begin we will prove an equivariant version of the generalized Schoenflies theorem using a Bing-shrinking argument.
In the non equivariant setting, the result is originally due to Brown~\cite{BrownSchoenflies} and Mazur~\cite{MazurSchoenflies}, but we will mostly  follow Putman's exposition~\cite{Putman} of Brown's proof.

The first step toward this is an equivariant collaring theorem. 
The argument follows closely~\cite[Theorem 2.1]{Putman} in the non-equivariant case.
Putman's exposition follows Connelly's proof of the collaring theorem~\cite{Connelly} which is originally due to Brown~\cite{BrownCollars}.

\begin{theorem}[Equivariant collaring theorem]
	 \label{thm:equivariantcollaring}
If~$M$ is a~$4$-manifold with boundary~$\partial M = Y$ and~$\rho\colon M \to M$ is an order 2 homeomorphism,
then~$Y$ has a~$\rho$-equivariant collar neighborhood,  that is an embedding~$i\colon Y \times [0,1] \hookrightarrow M$ with~$i(\rho(y),t) = \rho(i(y,t))$ and~$i(Y\times\{0\}) = \partial M$.
\end{theorem}
\begin{proof}
We follow the proof in the non-equivariant setting given by Connelly \cite{Connelly}; see also~\cite[Theorem 2.1]{Putman}. 
The proof has two steps: firstly we construct a space~$N$ that contains a subspace~$N_{-1}$ that is equivariantly homeomorphic to~$M$, and secondly we construct an equivariant homeomorphism~$\sigma\colon N \to N$ such that~$\sigma(M) = N_{-1}$.

We begin with the first step.
Let~$\{U_1, \dots , U_k\}$ be an open cover of~$Y$ such that for each~$U_i$ there is an embedding~$f_i\colon U_i \times [0,1] \to M$. 
Applying~$\rho$ to each~$U_i$ and~$f_i$,  leads to a~$\rho$-invariant open cover of~$Y$, namely
\[
\{U_1, \rho U_1, U_2, \rho U_2, \dots, U_k, \rho U_k\}.
\] 
Consider also the embeddings~$g_i \colon \rho U_i \times [0,1] \to M$ given by~$g_i(u,t) = \rho f_i(\rho(u),t)$. 
Choose a partition of unity~$\{\alpha_i \colon U_i \to [0,1]\}$ subordinate to the~$U_i$ and note that~$\{\alpha_i \circ \rho \colon \rho U_i \to [0,1]\}$ is a partition of unity subordinate to the~$\rho U_i$. 
Define a~$\rho$-invariant partition of unity subordinate to~$\{U_1, \rho U_1, U_2, \rho U_2, \dots, U_k, \rho U_k\}$ by~$\{\frac{1}{2}\alpha_i, \frac{1}{2}\alpha_i \circ \rho\}$. Write~$\beta_i = \frac{1}{2}\alpha_i$ and~$\overline{\beta}_i = \frac{1}{2}\alpha_i \circ \rho$ for brevity.

Define~$N$ as the result of gluing~$Y \times (-\infty, 0]$ to~$M$ by~$(m,0) \sim m$ for all~$m \in Y$. Note that this is equivariant with respect to~$\rho \times \id$ on~$Y \times (-\infty, 0]$. Let~$N_{-1}$ be the subspace~$M \cup (Y \times [-1,0])$ of~$N$. 
This concludes the first step of the proof.

It remains to construct an equivariant homeomorphism~$\sigma\colon N \to N$ such that~$\sigma(M) = N_{-1}$.
First extend~$f_i$ to the natural embedding~$F_i \colon U_i \times (-\infty,1] \to N$, and similarly~$g_i$ to the natural embedding~$G_i \colon \rho U_i \times (-\infty,1] \to N$. 
Then define a function~$\zeta_a \colon (-\infty,1] \to (-\infty,1]$ by
\[
\zeta_a(t) = \begin{cases}
t-a, &-\infty < t \leq 0,\\
(1+2a)t - a, &0\leq t \leq 1/2,\\
t, &1/2 \leq t \leq 1.
\end{cases}
\]
For each~$i$, we now define~$\widehat{\sigma}_i \colon U_i \times (-\infty,1] \to U_i \times (-\infty,1]$ by 
\[
\widehat{\sigma}_i(u,t) = (u,\zeta_{\beta_i(u)}(t)).
\]
Since~$\widehat{\sigma}_i$ is the identity outside of~$U_i \times (-\infty,1]$, we can extend it by the identity to a homeomorphism~$\sigma_i \colon N \to N$. Symmetrically, we can define homeomorphisms~$\overline{\sigma}_i = \rho \circ \sigma_i \circ \rho \colon N \to N$ which are identity away from~$\rho U_i \times (-\infty,1]$.
 The composition 
\[
\sigma_1 \circ \overline{\sigma}_1 \circ \sigma_2 \circ \overline{\sigma}_2 \circ \dots \circ \sigma_k \circ \overline{\sigma}_k \colon N \to N
\]
then restricts to the desired equivariant homeomorphism~$M \to N_{-1}$.
\end{proof}

Recall that $\rho_{std}$ always refers to a linear involution on $B^4$, or to the corresponding action on~$S^4$ given by quotienting $(B^4, \rho_{std})$ by its boundary.  

\begin{definition}
\label{def:equivariantly cellular}
Let~$M$ be a 4-manifold, and~$\rho$ an involution on~$M$. Let~$X$ be a~$\rho$-invariant subset of~$M$. Then~$X$ is \emph{equivariantly cellular} if for any open neighborhood~$U \supseteq X$ there is a sequence of sets~$\{C_i\}$ satisfying
\begin{itemize}
  \item for all~$i$, there is an equivariant homeomorphism~$(C_i,\rho \mid_{C_i}) \cong (B^4, \rho_{std})$ for some linear involution $\rho_{std}$,
  \item for all~$i$,~$C_i \subseteq U$,
  \item for all~$i$,~$C_{i+1}$ is contained in the interior of~$C_i$, and
  \item~$\displaystyle \bigcap_i C_i = X$.
\end{itemize}
\end{definition}

\begin{remark}
\label{rem:Cellular}
We record some remarks on Definition~\ref{def:equivariantly cellular}.
\begin{itemize}
\item In the nonequivariant setting,  the condition involving the open set~$U$ is typically not included in the definition of a cellular set (see e.g.~\cite{BrownSchoenflies} or~\cite[Definition 3.14]{DET}) and can in fact be omitted as proved in~\cite[Proposition 3.16]{DET}.
We include the corresponding condition in the equivariant setting to simplify the proof of Lemma~\ref{lem:shrinkingcells}.
Indeed Putman's exposition does include this condition in the definition of cellularity~\cite[page 5]{Putman}.
\item For any~$\rho$-invariant subset~$X$ of~$M$,  the involution~$\rho$ induces an action on the quotient space~$M/X$ obtained by collapsing~$X$.
\item Any equivariantly cellular subset $X$ must have a fixed point, since any linear action on $B^4$ fixes $0$.
\end{itemize}
\end{remark}

The following lemma is an equivariant analogue of~\cite[Lemma 4.1]{Putman}.

\begin{lemma} 
\label{lem:shrinkingcells}
Let~$M$ be a compact 4-manifold,  let~$\rho \colon M \to M$ be an involution. 
If~$X_1, \dots,  X_s$ are equivariantly cellular and pairwise disjoint subsets of Int($M$), then~$M$ is homeomorphic to~$M'$,  the result of collapsing~$X_1, \dots, X_s$ to $s$ distinct points.
\end{lemma}
\begin{proof}
The proof exactly follows \cite[Lemma 4.1]{Putman}, making all objects symmetric, and maps equivariant. 
We sketch a rough idea of the proof here.

By induction, it is enough to show that~$M$ collapsed along~$X_1$ is equivariantly homeomorphic to~$M$. 
This is achieved by constructing an equivariant surjective map~$f \colon M \to M$ taking~$X_1$ to a point with~$f \mid_{M - X_1}$ injective. 
The map~$f$ is defined by taking the limit of equivariant maps~$\{f_i\}$, where~$f_1$ is the identity, and each other~$f_i$ contracts~$f_{i-1}(C_i)$ linearly and fixes point-wise the complement of~$f_{i-1}(C_{i-1})$. 
Note that a linear contraction is equivariant in~$(B^4, \rho_{std})$ for any linear involution $\rho_{std}$. 
This explains why Putman's~$g_i$ (and therefore the~$f_i$) are equivariant.
\end{proof}

The following lemma is an equivariant analogue of~\cite[Lemma 4.2]{Putman}.

\begin{lemma} 
\label{lem:cellular}
Let $\rho_{std}$ be a linear involution on $B^4$. 
Let~$X_1, \dots, X_s$ be pairwise disjoint closed subsets of \textup{Int}$(B^4, \rho_{std})$ such that each~$X_i$ is~$\rho_{std}$-invariant. Define~$(M',\rho')$ to be the result of collapsing~$X_1, \dots, X_s$ and let~$\pi \colon B^n \to M'$ be the collapse map. 
If there exists an equivariant embedding~$(M',\rho') \to (S^4,\rho_{std})$ that takes~$\pi(\textup{Int}(B^4)) \subseteq M'$ to an open subset of~$S^4$, then each~$X_i$ is equivariantly cellular.
\end{lemma}
\begin{proof}
The proof exactly follows \cite[Lemma 4.2]{Putman}, making all objects symmetric (in particular the chosen points should be fixed points), and maps equivariant. Using the notation from \cite{Putman}, the only additional thing to check in the equivariant setting is that the map~$h_j$ is equivariant, which is immediate since~$B^n$ and~$B_{\varepsilon/j}$ both have the standard symmetry. 
\end{proof}

We can now prove an equivariant~$4$-dimensional Schoenflies theorem.
For this, recall that we simultaneously use~$\rho_{std}$ to denote a linear involution on $\R^4$, its restriction to~$B^4$, and the involution on~$S^4$ given by compactification.

\begin{customthm}
{\ref{thm:EquivariantSchoenfliesIntro}}
[Equivariant Schoenflies Theorem] 
\label{thm:eqSchoenflies}
Let~$\rho_{std} \colon S^4 \to S^4$ be a standard linear involution and let~$\rho_3 \colon S^3 \to S^3$ be 
a linear involution that satisfies~$
\dim \Fix(\rho_3) = \dim \Fix(\rho_{std})-1$.
If~$f \colon (S^3,\rho_3) \to (S^4, \rho_{std})$ is an equivariant locally flat embedding, then the closure of each component of~$S^4 - f(S^3)$  is equivariantly homeomorphic to~$(B^4, \rho_{std})$.
\end{customthm}
\begin{proof}
The proof is split into two cases. 
First, we consider the case when~$\rho_{std}$ has a 0-dimensional fixed-point set; that is~$\Fix(\rho_{std}) = S^0$. 
Without loss of generality, we can work with the closure~$A$ of one component of~$S^4 - f(S^3)$. 
Now choose a tubular neighborhood~$\overline\nu(x_0)$ of the fixed point of~$\rho_{std}$ in~$A$. 
Schoenflies' theorem ensures $A$ is homeomorphic to a $4$-ball.
Then by the annulus theorem we have that~$A - \nu(x_0) \cong S^3 \times I$, and by Theorem \ref{thm:KwasikSchultz}, the restriction of~$\rho_{std}$ to~$A - \nu(x_0)$ is conjugate to~$\rho_3 \times \id$ so that~$A$ is equivariantly homeomorphic to~$\nu(x_0) \cong (B^4, \rho_{std})$ glued to~$(S^3 \times I, \rho_3 \times \id)$, which is equivariantly homeomorphic to~$(B^4, \rho_{std})$.

Second, we consider the case when the fixed-point set of~$\rho_{std}$ is 1, 2, or 3-dimensional. 
Let~$A_1$ and~$A_2$ be the closures of the two components of~$S^4 - f(S^3)$. 
By the generalized Schoenflies theorem we have that the $A_i$ are (nonequivariantly) homeomorphic to $4$-balls. 
Since $f$ is an equivariant embedding, the involution restricts to $A_1 \sqcup A_2$.
Since~$\dim \Fix(\rho_3) =\dim \Fix(\rho)-1$,  one of $A_1$ or $A_2$ contains a fixed point so that $\rho$ does not swap $A_1$ and $A_2$.
\color{black}
Thus it remains to show that~$A_1$ has the standard symmetry, and the same argument will apply to~$A_2$.

Applying Theorem \ref{thm:equivariantcollaring} yields an equivariant collar neighborhood~$(N,\rho) \cong (S^3 \times I, \rho_3 \times \id) \subseteq A_1$ of the boundary of~$A_1$. 
Let~$X_1$ be the closure of~$A_1 - N$, and let~$X_2 = A_2$. 

We assert that $X_1$ and $X_2$ are equivariantly cellular.
Since the fixed-point set of~$\rho_{std}$ is at least 1-dimensional, it includes
points in~$N$ (if the fixed-point set lay entirely in~$A_1 \cong B^4$,  since the~$A_i$ are $\rho$-invariant,~$A_2 \cong B^4$ would be admit a free involution, contradicting the Brouwer fixed-point theorem) so that we can remove a small ball~$B \cong (B^4, \rho_{std})$ contained in~$N$ such that its complement~$M = S^4 - B$ is a 4-ball with the standard symmetry. Here we used that $\rho_3$ is linear to find such a $B$.
Note also that~$X_1$ and~$X_2$ are~$\rho_{std}$-invariant disjoint closed subsets of~$M$. 
Let~$M'$ be the result of collapsing~$M$ along~$X_1 \cup X_2$.
Then~$M'$ naturally embeds into the manifold obtained from~$N$ by collapsing its two boundary components:
$$
M':=M/(X_1 \cup X_2) \hookrightarrow N/\partial N \cong (S^4, \rho_{std}).
$$
Thus, Lemma~\ref{lem:cellular} implies that~$X_1$ and~$X_2$ are equivariantly cellular, as asserted.

Since~$X_1$ is equivariantly cellular,  Lemma~\ref{lem:shrinkingcells} implies that~$A_1$ is equivariantly homeomorphic to~$A_1$ collapsed along~$X_1$.
The latter is equivariantly homeomorphic to~$N$ collapsed along~$S^3 \times \{0\}$ which is~$(B^4, \rho_{std})$. 
\end{proof}


We now prove Theorem~\ref{thm:1dimlclassification} which asserts that there is a unique fixed-point linear involution~$S^4 \to S^4$ with fixed-point set homeomorphic to a circle, up to conjugation in the homeomorphism group of~$S^4$; this involution is~$\rho_0$.

\begin{proof}[Proof of Theorem \ref{thm:1dimlclassification}]
Let~$\rho$ be a fixed-point linear involution on~$S^4$ with a 1-dimensional fixed-point set. 
Since $\rho$ is fixed-point linear,  its fixed-point set admits an equivariant tubular neighborhood $\nu$.
Set $X:=S^4 - \nu$.
If $\ks(X/\rho)=0$, then Theorem~\ref{thm:quotientclassification} (and Remark~\ref{rem:quotientclassification}) implies that $\rho$ is conjugate to $\rho_0$.
We therefore assume that $\ks(X/\rho)=1$.
Glue back in~$(\overline{\nu} - (I_1 \times D^3),\rho)$ to $(X,\rho)$ to get~$(B^4, \rho)$, where we take~$S^1 = I_1 \cup I_2$ to be a decomposition of~$S^1$ into two intervals. We now claim that the double of~$(B^4, \rho)$ is~$(S^4, \rho_{std})$. Then by the equivariant Schoenflies theorem (Theorem \ref{thm:eqSchoenflies}) we will have that~$(B^4, \rho)$ is equivariantly homeomorphic to~$(B^4, \rho_{std})$. 
We then have that~$(S^4, \rho)$ is obtained from~$(B^4, \rho_{std})$ by gluing in~$(\overline{\nu} - (I_2 \times D^3), \rho_{std}) \cong (B^4, \rho_{std})$. 
Since every equivariant homeomorphism on~$S^3$ extends to~$B^4$ by taking the cone construction,
this implies that~$(S^4, \rho) \cong (S^4, \rho_{std})$, as desired. 

To see that the double of~$(B^4, \rho)$ is~$(S^4, \rho_{std})$, note that the double is obtained by gluing two copies of~$(X,\rho)$ along the remaining~$S^2 \times I = \partial X \cap \partial B^4$ boundaries to obtain~$(S^2 \times D^2,\rho)$, and then filling in the remaining~$D^3 \times S^1 = (D^3 \times I_1) \cup (D^3 \times I_1)$. 
Here, note that~$D^3 \times S^1$ has the standard symmetry~$\rho_{std}$ since the $D^3 \times I_1$ inherit the equivariant framing of~$\nu$.
 Since the Kirby-Siebenmann invariant is additive under boundary gluing (see e.g.\cite[Theorem~9.2]{FriedlNagelOrsonPowell}),  we have~$\ks((S^2 \times D^2)/\rho) = \ks(X/\rho) + \ks(X/\rho) = 1+1 = 0$, and hence~$(S^2 \times D^2, \rho) \cong (S^2 \times D^2, \rho_{std})$ by Theorem \ref{thm:quotientclassification}. 
As before, gluing in the remaining~$(S^1 \times D^3, \rho_{std})$,  either by the identity map or a Gluck twist, produces~$(S^4, \rho_{std})$, as desired.
\end{proof}

We now prove Theorem~\ref{thm:nonuniqueness} which asserts that the fixed-point set of the involution~$\rho$ in Theorem~\ref{thm:1dimlclassification} has exactly two equivariant tubular neighborhoods~$\nu_0,\nu_1$ up
up to equivariant homeomorphism of 
pairs~$(S^4, \nu)$. 

\begin{proof}[Proof of Theorem \ref{thm:nonuniqueness}]
We first observe that by Theorem~\ref{thm:quotientclassification},  there are involutions $\rho_0,\rho_1$ of $S^4$ and tubular neigbhorhoods $\nu_0,\nu_1$ of the fixed-point sets with~$\ks(X_{\rho_0,\nu_0})=0$, and~$\ks(X_{\rho_1,\nu_1})=1$.
Theorem~\ref{thm:1dimlclassification} implies that these two involutions are conjugate.
As a consequence,  we deduce that for the involution $\rho$ from Theorem~\ref{thm:1dimlclassification}, there are two equivariant tubular neighborhoods~$\nu_0$ and~$\nu_1$ of the fixed-point set which have complements which are not equivariantly homeomorphic (indeed~$S^4 - \nu_0$ has a quotient with~$\ks = 0$ and~$S^4 - \nu_1$ has a quotient with~$\ks = 1$). 

We now show that there are at most two such equivariant tubular neighborhoods, up to equivariant homeomorphism of pairs. 
Given any equivariant tubular neighborhood $\nu$, we can apply Theorem \ref{thm:quotientclassification} to see that $S^4 - \nu$ is equivariantly homeomorphic to either $S^4 - \nu_0$ or $S^4 - \nu_1$, which can be taken relative to the boundary, possibly up to a Gluck twist. Finally, we can extend the homeomorphism on the boundary to an equivariant homeomorphism betweeen $\nu$ and either $\nu_0$ or $\nu_1$, since the Gluck twist extends equivariantly over $\nu$. The result is an equivariant homeomorphism of pairs $(S^4, \nu) \cong (S^4, \nu_0)$ or $(S^4,\nu) \cong (S^4, \nu_1)$, as desired.
\end{proof}

\section{Modified surgery}
\label{sec:Modified}

This section introduces the surgery theoretic background needed to prove Theorem~\ref{thm:quotientclassification}.
Section~\ref{sub:Normal1Type} reviews normal smoothings,  whereas Section~\ref{sub:Monoid} is concerned with the~$\ell$-monoid~$\ell_5(\Z[\pi],w)$, and Section~\ref{sub:Obstruction} concludes with an overview of the modified surgery obstruction.
Informally, whereas classical surgery is used to decide whether or not (simple) homotopy equivalent $n$-manifolds are $h$-cobordant (or~$s$-cobordant),  modified surgery attempts to reach the same conclusion while only assuming that the homotopy groups of the manifolds agree up to (at most) the middle dimension; this is made more formal using normal $k$-types and normal $k$-smoothings.
The trade-off for this less restrictive starting point is that the resulting surgery $\ell$-monoids are more unwieldy than Wall's $L$-groups.

\medbreak

Throughout this section, it will be helpful to recall the following terminology.
A map~$f\colon X\to Y$ is \emph{$m$-connected} if~$f_*\colon\pi_i(X)\to \pi_i(Y)$ is an isomorphism for~$i<m$ and is surjective for~$i=m$. 
A map~$f\colon X\to Y$ is \emph{$m$-coconnected} if~$f_*\colon\pi_k(X)\to \pi_k(Y)$ is an isomorphism for~$i>m$ and is injective for~$i=m$.

\subsection{Normal smoothings}
\label{sub:Normal1Type}
We briefly review normal smoothings and normal~$k$-types following~\cite{KreckSurgeryAndDuality}.
Let~$B$ be a  space with the homotopy type of a CW complex and let~$\xi \colon B \to \BTOP$ be a fibration.
An \emph{$n$-dimensional~$(B, \xi)$-manifold} consists of a pair
$(M,\overline{\nu})$, where~$M$ is an oriented~$n$-manifold
and~$\overline{\nu} \colon M \to B$ is a lift of the stable normal bundle~$\nu \colon M \to \BTOP$ of~$M$, meaning that~$\nu=\xi \circ \overline{\nu}$.

\begin{definition}
\label{def:NormalSmoothing}
Let~$B$ be a space with the homotopy type of a CW complex with finite~$(k+2)$-skeleton, let~$\xi \colon B \to \BTOP$ be a fibration, and let~$(M,\overline{\nu})$ be a~$(B, \xi)$-manifold.
\begin{enumerate}
\item If~$\overline{\nu}$ is~$(k{+}1)$-connected, then~$(M, \overline{\nu})$ is called a \emph{normal~$k$-smoothing} into~$(B,\xi)$.
\item The pair~$(B,\xi)$ is a \emph{normal~$k$-type} for~$M$ if~$\xi$ is~$(k+1)$-coconnected and there exists a normal~$k$-smoothing~$(M, \overline{\nu})$ into~$(B,\xi)$.
\end{enumerate}
\end{definition}

Next, a~\emph{$(B, \xi)$-nullbordism} for a closed~$n$-dimensional~$(B, \xi)$-manifold~$(M_0,\overline{\nu}_0)$ is an~$(n{+}1)$-dimensional~$(B, \xi)$-manifold~$(W,\overline{\nu})$ for which~$\partial (W,\overline{\nu})=(M_0,\overline{\nu}_0)$.
A~\emph{$(B, \xi)$-cobordism} between two~$n$-dimensional~$(B, \xi)$-manifolds~$(M_0,\overline{\nu}_0)$
and~$(M_1,\overline{\nu}_1)$ essentially consists of a homeomorphism~$f \colon \partial M_0 \to \partial M_1$ and an~$(n{+}1)$-dimensional~$(B, \xi)$-nullbordism of~$(M_0 \cup_f M_1, \overline{\nu}_0 \cup \overline{\nu}_1)$; we refer to~\cite[Section~8.1]{ConwayOrsonPowell} for a more detailed discussion.

\begin{example}
\label{ex:Normal1Type}
We conclude with the setting that will be of interest in Section~\ref{sec:ProofTechnical}.
The normal~$1$-type of a~$\Pin^+$ nonorientable~$4$-manifold~$M$ with~$\pi_1(M) \cong \Z/2$ is~$\xi \colon \BTOPPin^+ \to \BTOP$~\cite[Below Lemma~2.2]{HambletonKreckTeichnerNonorientable}
(in order to see that $\BTOPPin^+$ admits a model with finite $3$-skeleton, consider the fibration $B\Z_2 \to \BTOPPin^+ \to \BTOP$,  deduce that $\pi_1(\BTOPPin^+)\cong \Z_2$ and $\pi_i(\BTOPPin^+)=0$ for $i=2,3$ and conclude by applying (the proof of) CW approximation).
A normal~$1$-smoothing arises from the choice of a (normal)~$\Pin^+$ structure or equivalently from the choice of a tangential~$\Pin^-$ structure~\cite[Convention on page~654]{HambletonKreckTeichnerNonorientable}.
Writing~$B:=\BTOPPin^+$ for brevity,  it is known that taking the Kirby-Siebenmann invariant determines an isomorphism~$\Omega_4(B,\xi) \cong \Z/2,(M,\overline{\nu}) \mapsto \ks(M)$~\cite[page~654]{HambletonKreckTeichnerNonorientable}.
\end{example}

\subsection{The~$\ell$-monoid}
\label{sub:Monoid}
This section gives a brief overview of some facts related to the~$\ell_5$-monoid from~\cite{KreckSurgeryAndDuality,CrowleySixt}.
Our exposition follows closely~\cite{ConwayOrsonPowell,ConwayOrson}. 
We assume familiarity with quadratic forms and their lagrangians; see e.g.~\cite[pages 253 and 260]{LueckMacko}.
In what follows,  all modules are assumed to be stably free.

\begin{notation}
From now on, we fix a group~$\pi$ such that the Whitehead group~$\operatorname{Wh}(\pi)$ is trivial and do not concern ourselves with Whitehead torsion.
In the next section, we work with~$\pi = \Z/2$, a group for which this assumption is satisfied.
We also fix a homomorphism~$w \colon \pi \to C_2=\{ \pm 1\}$ and endow the ring~$\Z[\pi]$ with the involution
$$
\overline{\sum_g n_gg}:=\sum_g n_g w(g)g^{-1}.
$$
\end{notation}

The monoid~$\ell_5(\Z[\pi],w)$ consists of equivalence classes of triples~$((M,\psi);F,V)$ where~$(M,\psi)$ is a quadratic form over~$\Z[\pi]$,~$F \subseteq  M$ is a lagrangian and~$V \subseteq  M$ is a half-rank direct summand.
These triples are called \emph{quasiformations}.
An element~$x \in \ell_5(\Z[\pi],w)$ is called \emph{elementary} if it is represented by a quasiformation~$((M,\psi);F,V)$ where the inclusions induce an isomorphism~$F \oplus V \cong M$.
Two quadratic forms~$(P,\psi)$ and~$(P',\psi')$ are \emph{$0$-stably equivalent} if there are zero forms~$(Q,0)$ and~$(Q',0)$ and an isometry 
$$(P,\psi) \oplus (Q,0) \cong (P',\psi') \oplus (Q',0).$$
For every quadratic form~$v'=(V',\theta')$ over~$\Z[\pi]$,  we write~$\ell_5(v') \subseteq  \ell_5(\Z[\pi],w)$ for the subset of equivalence classes of quasiformations~$((M,\psi);F,V)$ such that the \emph{induced forms}~$(V,\psi|_{V \times V})$ and~$(V^\perp,\psi|_{V^\perp \times V^\perp})$ are~$0$-stably equivalent to~$(V',\theta')$.
For every~$0$-stable equivalence class of a quadratic form~$v'=(V',\theta')$,  the subset~$\ell_5(v')\subset \ell_5(\Z[\pi],w)$ contains an elementary class~\cite[Corollary 5.3]{CrowleySixt}; in particular it is nonempty.

We specialise to~$\pi=\Z/2$ and record the following result that is due to Hambleton-Kreck-Teichner~\cite[Proposition 4 item (c)]{HambletonKreckTeichnerNonorientable}.
\begin{proposition}
\label{prop:HKTProp4}
If~$v=(\Z[\Z/2],\theta)$ is a quadratic form whose symmetrization is the zero form, then every element of~$\ell_5(v) \subset \ell_5(\Z[\Z/2],-)$ is elementary.
\end{proposition}

The next section describes the relevance of this result to~$4$-manifold topology.

\subsection{The modified surgery obstruction.}
\label{sub:Obstruction}

This section briefly summarises what is known about Kreck's obstruction to a~$(B,\xi)$-cobordism being~$(B,\xi)$-bordant rel.~boundary to an~$s$-cobordism.
We refer to~\cite[Sections 6 and 7]{KreckSurgeryAndDuality} for more details as well as to~\cite[Sections 7 and 8]{ConwayOrsonPowell} and~\cite[Section 4]{ConwayOrson} for expositions resembling this one.

\begin{notation}
Fix~$4$-manifolds~$M_0$ and~$M_1$ with normal~$1$-type~$(B,\xi)$ and a~$(B,\xi)$-cobordism~$(W,\overline{\nu})$ between~$(M_0,\overline{\nu}_0)$ and~$(M_1,\overline{\nu}_1)$.
For brevity, we set~$\pi:=\pi_1(B)$ and~$w:=w_1(M_i)$.
\end{notation}

Kreck~\cite[Theorem 4]{KreckSurgeryAndDuality} defines a class~$\Theta(W,\overline{\nu}) \in \ell_5(\Z[\pi],w)$, called the \emph{modified surgery obstruction}, that only depends on the~$(B,\xi)$-bordism rel. boundary class of~$(W,\overline{\nu})$ and that is elementary if and only if~$(W,\overline{\nu})$ is~$(B,\xi)$-bordant rel.~boundary to an~$s$-cobordism. 
We describe how this obstruction relates to the algebra of the previous section. 
Write 
$$K\pi_2(M_i):=\ker((\overline{\nu}_i)_* \colon \pi_2(M_i) \to \pi_2(B)).$$
Intersections and self-intersections in~$M_i$ define a quadratic form~$(K\pi_2(M_i),\psi_{M_i})$ called the \emph{Wall form} of~$(M_i, \overline{\nu}_i)$~\cite[Section 5]{KreckSurgeryAndDuality}.
Write~$\lambda_{M_i}$ for the symmetrization of~$\psi_{M_i}$ (this agrees with the restriction of the equivariant intersection form on~$\pi_2(M_i) \cong H_2(M_i;\Z[\pi])$) and
$$\rad(\lambda_{M_i})=\{ x \in K\pi_2(M_i) \mid \lambda_{M_i}(x,y)=0 \text{ for all } y \in K\pi_2(M_i) \}.$$
Given a free~$\Z[\pi]$-module~$S_i$ and a surjection~$\varpi_i \colon S_i \twoheadrightarrow K \pi_2 (M_i)/\rad(\lambda_{M_i})$, the pull-back of the (nondegenerate) Wall form~$(K\pi_2(M_i)/\rad(\lambda_{M_i}),\psi_{M_i})$ by~$\varpi_i$ is called a {\em free Wall form} of~$(M_i, \overline{\nu}_i)$.

As explained in~\cite{CrowleySixt} (see also~\cite[Section~8]{ConwayOrsonPowell}),  another result of Kreck~\cite[Proposition~8]{KreckSurgeryAndDuality} implies that for any choice of free Wall forms~$v(\overline{\nu}_0)$ and~$v(\overline{\nu}_1)$ for~$(M_0,\overline{\nu}_0)$ and~$(M_1,\overline{\nu}_1)$ that are~$0$-stably isometric to a form~$v'$,
we have
\begin{equation}
\label{eq:ObstructionLocation}
\Theta(W,\overline{\nu}) \in \ell_{5}(v').
\end{equation}
The next proposition applies this fact in conjunction with Proposition~\ref{prop:HKTProp4}.
\begin{proposition}
\label{prop:ApplyHKTAlgebra}
Let~$M_0$ and~$M_1$ be nonorientable~$4$-manifolds with~$\pi_1(M_i) \cong \Z/2$ and normal~$1$-type $(B,\xi)$ and let~$f \colon \partial M_0 \to \partial M_1$ be a homeomorphism.
If~$M_0$ and~$M_1$ are bordant over~$(B,\xi)$ relative to~$f$ and have free Wall forms that are $0$-stably isometric to a quadratic form~$v:=(\Z[\Z/2],\theta)$ whose symmetrization is trivial, then~$f$ extends to a homeomorphism~$M_0 \to M_1$.
\end{proposition}
\begin{proof}
By assumption,  there is a~$(B,\xi)$-cobordism~$(W,\overline{\nu})$ relative to~$f$.
As described in~\eqref{eq:ObstructionLocation}, the modified surgery obstruction~$\Theta(W,\overline{\nu}_W)$ defines an element of~$\ell_5(v)$.
Proposition~\ref{prop:HKTProp4} ensures that~$\Theta(W,\overline{\nu}_W)$ is elementary.
The work of Kreck then implies that~$(W,\overline{\nu}_W)$ is~$(B,\xi)$-bordant rel. boundary to an~$s$-cobordism~\cite[Theorem 4]{KreckSurgeryAndDuality}.
Since~$\Z/2$ is a good group, the proposition now follows from the~$5$-dimensional~$s$-cobordism theorem~\cite[Theorem 7.1A]{FreedmanQuinn}.
\end{proof}

\section{$1$-dimensional fixed-point sets: conclusion of the proof}
\label{sec:ProofTechnical}

The goal of this section is to prove Theorem~\ref{thm:quotientclassification} which was used during the proof of Theorem~\ref{thm:1dimlclassification} to classify fixed-point linear involutions~$S^4 \to S^4$ with fixed-point set~$\Fix(\rho_i)$ homeomorphic to a circle.
To do so, we will follow the modified surgery program that was described in Section~\ref{sec:Modified}.
Specifically, we seek to apply Proposition~\ref{prop:ApplyHKTAlgebra}.

\subsection{The topology of the orbit set}

Let~$\psi \colon \pi_1(S^1 \times \R P^2) \to \Z/2$ be the surjection that sends the~$S^1$-factor to zero.

\begin{lemma}
\label{lem:WellDef}
Let~$\rho \colon S^4 \to S^4$ be an orientation-reversing involution with fixed-point set~$\Fix(\rho)$ homeomorphic to~$S^1$ and equivariant tubular neighborhood~$\nu$.
The orbit set~$X:=(S^4 - \nu)/\rho$ is nonorientable and satisfies the following properties:
\begin{enumerate}
\item The boundary of~$X$ is homeomorphic to~$S^1 \times \R P^2$.
\item The fundamental group of~$X$ has order~$2$.
\item 
Under any identification~$\partial X \cong S^1 \times \R P^2$,  the inclusion induced map~$\pi_1(\partial X) \to \pi_1(X)$ agrees with~$\psi.$
\item The universal cover~$\widetilde{X}$ of~$X$ is such that the inclusion~$\partial \widetilde{X} \to \widetilde{X}$ induces an isomorphism
$$ \Z \cong  H_2(\partial \widetilde{X}) \to H_2(\widetilde{X}).$$
In particular, the $\Z$-intersection form of $\widetilde{X}$ and equivariant intersection form of $X$ vanish:
$$Q_{\widetilde{X}}=0\quad \text{and} \quad \lambda_X=0.$$
\end{enumerate}
\end{lemma}
\begin{proof}
We prove the first assertion. 
Restricting an equivariant framing
to~$\partial \overline{\nu}=\partial(S^4 - \nu)$ gives an equivariant homeomorphism~$\partial(S^4 - \nu) \to S^1 \times S^2$ which  descends to a homeomorphism
$$\partial X \xrightarrow{\cong} S^1 \times \R P^2.$$
We prove the second and third assertions.
The long exact sequence of homotopy groups associated to the double covers~$S^2 \times D^2 \cong \widetilde{X} \to X$ and~$S^2 \times S^1 \cong \partial \widetilde{X} \to \partial X$ gives rise to the following commutative diagram:
$$
\xymatrix{
1 \ar[r]& {\overbrace{\pi_1(\widetilde{X})}^{=1}} \ar[r]& \pi_1(X) \ar[r]^-\cong& \Z/2 \ar[r]& 1 \\
1 \ar[r]& \pi_1(\partial \widetilde{X}) \ar[r]\ar[u]^{i_*}& \pi_1(\partial X) \ar[r]\ar[u]^{i_*}& \Z/2 \ar[r]\ar[u]^= & 1.
}
$$
The top row of this diagram shows that~$\pi_1(X) \cong \Z/2.$
The commutativity of the diagram ensures that the inclusion induced map~$i_* \colon \pi_1(\partial X) \to \pi_1(X)$ is surjective and agrees with~$\psi.$

The last assertion follows because the exact sequence of the pair~$(\widetilde{X},\partial \widetilde{X}) \cong (S^2 \times D^2,S^2 \times S^1)$ shows that~$H_2(\partial \widetilde{X}) \to H_2(\widetilde{X})$ is an isomorphism.
\end{proof}

\subsection{Proof of Theorem \ref{thm:quotientclassification}.}
\label{sub:ModifiedProof}
Let~$\rho_0,\rho_1 \colon S^4 \to S^4$ be involutions with fixed-point set homeomorphic to a circle, and let~$\nu_0,\nu_1$ be equivariant tubular neighborhoods of the respective fixed-point sets.
Recall that Theorem~\ref{thm:quotientclassification} states that there is a homeomorphism~$X_{\rho_0,\nu_0} \cong X_{\rho_1,\nu_1}$ if and only if~$\ks(X_{\rho_0,\nu_1}) = \ks(X_{\rho_1,\nu_1})$; both values of the Kirby-Siebenmann invariant are realized.

\begin{proof}[Proof of Theorem~\ref{thm:quotientclassification}]
If there is a homeomorphism~$X_{\rho_0,\nu_0} \cong X_{\rho_1,\nu_1}$,  then~$\ks(X_{\rho_0,\nu_0})=\ks(X_{\rho_1,\nu_1})$.
We therefore focus on the converse and apply the modified surgery program from~\cite{KreckSurgeryAndDuality} (and specifically Proposition~\ref{prop:ApplyHKTAlgebra}) to show that~$X_{\rho_0,\nu_0}$ and~$X_{\rho_1,\nu_1}$ are homeomorphic.

More precisely, given equivariant framings~$\fr_0\colon \overline{\nu}_0 \to S^1 \times D^3$ and~$\fr_1 \colon \overline{\nu}_1 \to S^1 \times D^3$,  we show that if~$\ks(X_{\rho_0,\nu_0}) = \ks(X_{\rho_1,\nu_1})$,  then there is an equivariant homeomorphism $S^4 - \nu_0 \to S^4 - \nu_1$ extending either $\fr_1^{-1}\vert_{\partial}\circ \fr_0\vert_{\partial}$ or $\fr_1^{-1}\vert_{\partial}\circ G \circ \fr_0\vert_{\partial}$, where $G$ is the Gluck twist. 
Note that, passing to the orbit sets, this implies that~$X_{\rho_0,\nu_0}\cong X_{\rho_1,\nu_1}$.

Consider the homeomorphism~$\fr_1^{-1}\vert_{\partial}\circ \fr_0\vert_{\partial} \colon \partial \overline{\nu}_0 \to \partial \overline{\nu}_1$.
Taking the union of~$S^4 - \nu_0 \cong S^2 \times D^2$ and~$S^4 - \nu_1 \cong S^2 \times D^2$ along this homeomorphism may either result in~$S^2 \times S^2$ or~$S^2 \widetilde{\times} S^2$.
If the former occurs,  set~$\varphi:=\fr_1^{-1}\vert_{\partial} \circ \fr_0\vert_{\partial}$ whereas if the latter occurs,  set~$\varphi:=\fr_1^{-1}\vert_{\partial} \circ G \circ \fr_0\vert_{\partial}$ where~$G$ denotes the Gluck twist.

Note that $\varphi$ is equivariant because $G,\fr_0$ and $\fr_1$ are equivariant.
Also, by construction, we deduce that~$S^4- \nu_0 \cup_{\varphi} S^4 - \nu_1$ is spin; in fact it is homeomorphic to~$S^2 \times S^2$.

Since the homeomorphism~$\varphi$ is equivariant it induces a homeomorphism~$f \colon \partial X_{\rho_0,\nu_0} \to \partial X_{\rho_1,\nu_1}$.
By construction of $\varphi$, the universal cover of~$X:=X_{\rho_0,\nu_0} \cup_f X_{\rho_1,\nu_1}$ is spin and therefore~$X$ is either tangentially~$\TOPPin^-$ or tangentially~$\TOPPin^+$.
If it were tangentially~$\TOPPin^+$, then the boundary of the~$X_{\rho_i,\nu_i}$ would be tangentially~$\TOPPin^+$ for $i=0,1$.
This would imply that~$0=w_2(\R P^2 \times S^1)=w_2(\R P^2)$, a contradiction.
Thus~$X$ is tangentially~$\TOPPin^-$.

Since~$X_i:=X_{\rho_i,\nu_i}$ has fundamental group of order~$2$ and~$\pi_1(\partial X_i) \to \pi_1(X_i)$ is surjective (recall the second and third items of Lemma~\ref{lem:WellDef}) for $i=0,1$,  a van Kampen argument gives~$\pi_1(X) \cong \Z/2$.
As explained in~\cite[below Lemma 2]{HambletonKreckTeichnerNonorientable},  it follows that~$\xi \colon \BTOPPin^+ \to \BTOP$ is the normal~$1$-type of~$X$; the outline being that $\xi$ and the stable normal bundle $X \to \BTOP$ are $\pi_1$-isomorphism and $\pi_2(\BTOPPin^+)=0.$
Pick a normal~$1$-smoothing~$\alpha  \colon X \to \BTOPPin^+$.
It restricts to a map~$\alpha_i \colon X_i \to \BTOPPin^+$.
Again, the argument as in~\cite[below Lemma 2]{HambletonKreckTeichnerNonorientable} also shows these are necessarily~$2$-equivalences, and therefore normal~$1$-smoothings.

We assert that~$(X,\alpha)$ vanishes in the bordism group~$\Omega_4(B,\xi)$, where $B:=\BTOPPin^+$.
As recalled in Example~\ref{ex:Normal1Type}, it suffices to prove that $\ks(X)=0$, and this follows from the additivity of the Kirby-Siebenmann invariant (see e.g.~\cite[Proposition 8.2]{FriedlNagelOrsonPowell}) which implies that
$$\ks(X)=\ks(X_0)+\ks(X_1)=0.$$
Since~$(H_2(\widetilde{X}_i),Q_{\widetilde{X}_i}) \cong (\Z_-,0)$,  the equivariant intersection forms of~$X_0$ and~$X_1$ is trivial; recall the fourth item of Lemma~\ref{lem:WellDef}.
This implies that the $X_i$ admit a free Wall form of the form~$v=(\Z[\Z/2],\theta)$ whose symmetrization is trivial.
Proposition~\ref{prop:ApplyHKTAlgebra} therefore implies that~$f$ extends to a homeomorphism~$\Phi \colon X_0 \to X_1.$
Lifting to the covers,  and using the definition of $\varphi$, we conclude that there is an equivariant homeomorphism $S^4 - \nu_0 \to S^4 - \nu_1$ extending either $\fr_0^{-1}\vert_{\partial}\circ \fr_1\vert_{\partial}$ or $\fr_0^{-1}\vert_{\partial}\circ G \circ \fr_1\vert_{\partial}$.
Finally,  the actions that realise the values of the Kirby-Siebenmann are constructed in Constructions~\ref{cons:1dimKS=0} and~\ref{cons:1dimKS=1}.
\end{proof}

\begin{construction} \label{cons:1dimKS=0}
The smooth involution~$\rho_0$ was already constructed above Theorem~\ref{thm:1dimlclassification} by extending the involution~$\R^4 \to \R^4,(x_1,y_1,y_2,y_3) \mapsto (x_1,-y_1,-y_2,-y_3)$ over~$S^4$, but we briefly mention an equivalent definition which makes the orbit set $X_{\nu_0,\rho_0}$ more explicit.
To do so,  decompose~$S^4$ as the union of~$S^1 \times D^3$ and~$D^2 \times S^2$.
Consider the involution~$\id_{S^1} \times -\id_{D^3}$ on~$D^2 \times S^2$ and the involution~$\id_{D^2} \times \operatorname{antip}$ on~$S^2 \times D^2$, where~$\operatorname{antip}$ denotes the antipodal map (which is orientation-reversing on even-dimensional spheres).
These two involutions combine to form an involution~$\rho_{0} \colon S^4 \to S^4$ with~$\Fix(\rho_{0}) \cong S^1$ and equivariant tubular neighborhood~$\nu_0$ .
This construction agrees with the previous construction of $\rho_0$ but it is now apparent that~$X_{\rho_{0},\nu_0}$ satisfies~$\ks(X_{\rho_{0},\nu_0})=0.$
Indeed~$X_{\rho_{0},\nu_0} \cong (S^2 \times D^2)/\rho_{0} \cong \R P^2 \times D^2$ is smoothable.
\end{construction}

\begin{construction}
\label{cons:1dimKS=1}
Inspired by a construction from~\cite[p.~650-651]{HambletonKreckTeichnerNonorientable} (see also the exposition in~\cite[Section 5.4]{KasprowskiPowellRayCounterexamples}),
we construct a locally linear involution~$\rho_1 \colon S^4 \to S^4$ with fixed-point set~$\Fix(\rho_1)$ homeomorphic to a circle and an equivariant tubular neighborhood~$\nu_1$ such that~$\ks(X_{\rho_1,\nu_1})=1$.
Connect sum the identity map~$\id_{\R P^2 \times D^2}$ with the collapse map~$E_8 \to S^4$ to obtain a degree one normal map
$$ F \colon (\R P^2 \times D^2) \# E_8 \to \R P^2 \times D^2.$$ 
Note that~$F|_{\partial}=\id_{\R P^2 \times S^1}$.
The quadratic form on the surgery kernel~$K_n(F)$ is~$\psi_{E_8} \otimes \Z[\Z/2]$ where~$\psi_{E_8}$ denotes the quadratic form underlying the~$E_8$ form.
It follows that the surgery obstruction~$\sigma(F) \in L_4(\Z[\Z/2],-)$ is represented by the image of~$\psi_{E_8}$ under the map~$L_4(\Z) \to L_4(\Z[\Z/2],-)$.
This is the zero map~\cite[Chapter 13A, bottom of page 173]{WallSurgeryOnCompact}, so~$\sigma(F)=0$.
Thus after stabilising with hyperbolics forms, the quadratic form on the surgery kernel of~$F$ is hyperbolic:
\begin{equation}
\label{eq:Hyperbolic}
(\psi_{E_8} \otimes \Z[\Z/2]) \oplus H_+(\Z[\Z/2])^{\oplus k}\cong  H_+(\Z[\Z/2])^{\oplus 4+k}.
\end{equation}
Consider the degree one normal map obtained from~$F$ by taking connected sums with the collapse map~$S^2 \times S^2 \to S^4$:
$$F' \colon \overbrace{(\R P^2 \times D^2) \# E_8 \#^k S^2 \times S^2}^{:=M} \to \R P^2 \times D^2.$$ 
The quadratic form on~$K_n(F')$ is~$(\psi_{E_8} \otimes \Z[\Z/2]) \oplus H_+(\Z[\Z/2])^{\oplus k}$.
The preimage of the module underlying~$H_+(\Z[\Z/2])^{\oplus 4+k}$ under the isometry in~\eqref{eq:Hyperbolic} provides a submodule of~$H_n(M;\Z[\Z/2])$ on which equivariant intersections and (reduced) self-intersections vanish.
It follows from~\cite[Corollary 1.4]{PowellRayTeichner} that there is a~$4$-manifold~$X$ with
$$ (\R P^2 \times D^2) \# E_8 \#^k S^2 \times S^2 \cong X \# \#^{4+k} S^2 \times S^2.$$ 
The additivity of the Kirby-Siebenmann invariant (see e.g.~\cite[Theorem 9.2]{FriedlNagelOrsonPowell}) ensures that~$\ks(X)=1$.

Since we showed that the degree one normal map~$F \in \mathcal{N}(\R P^2 \times D^2,\partial)$ has vanishing surgery obstruction,  and performed surgeries in the interior to kill the surgery kernel,  the resulting homotopy equivalence~$X \to \R P^2 \times D^2$ agrees with~$F$ on the boundary; i.e.  the outcome agrees with~$\id_{\R P^2 \times S^1}$ on the boundary.

In particular~$\pi_1(X) \cong \Z/2$ and the universal (double) cover~$\widetilde{X}$ of~$X$ is endowed with a free involution that restricts to~$\operatorname{antip} \times \id_{S^1}$ on~$S^2 \times S^1$ on the boundary. Here $\operatorname{antip}$ denotes the antipodal map.
Note that~$\widetilde{X}$ is homotopy equivalent to~$S^2 \times D^2$ by lifting the homotopy equivalence from~$X$ to~$\R P^2 \times D^2$. 
Gluing in a tubular neighborhood~$\nu_1 := D^3 \times S^1$ then gives a homotopy equivalence~$\widetilde{X} \cup (D^3 \times S^1)  \to S^4$. 
Freedman's classification of simply-connected 4-manifolds therefore implies that~$\widetilde{X} \cup (D^3 \times S^1)$ is homeomorphic to~$S^4$~\cite{Freedman}.
Extending the deck transformation action on~$\widetilde{X}$ over~$\nu_1$ by~$-\id_{D^3} \times \id_{S^1}$ yields a locally linear involution~$\rho_1 \colon S^4 \to S^4$ with~$\Fix(\rho_1) \cong S^1$ such that the orbit set~$X_{\rho_1,\nu_1} \cong X$ has~$\ks(X_{\rho_1,\nu_1})=1.$
\end{construction}

\section{Involutions on $B^4$} \label{sec:D4}
In this section we state the classification theorem of involutions on $B^4$. 
Note that there are no free involutions on~$B^4$ by the Brouwer fixed-point theorem.

\begin{theorem}
\label{thm:D4involutions}
Locally linear involutions of the closed~$4$-ball are classified as follows.
\begin{enumerate}
\item [(0)] When the fixed-point set is~$0$-dimensional,  it consists of one point.
Up to conjugacy, there is a unique locally linear involution on~$B^4$ that fixes one point.
\item [(1)] When the fixed-point set is~$1$-dimensional,  it is a properly embedded arc.
There is a unique conjugacy class of locally linear involutions on~$B^4$ whose fixed-point set is one-dimensional and admits an equivariant tubular neigbhorhood.
\item [(2)] When the fixed-point set is~$2$-dimensional,  it is a disk.
The set of conjugacy classes of locally linear involutions on~$B^4$ with~$2$-dimensional fixed-point set is in bijection with the set of isotopy classes of locally flat and properly embedded disks~$D \subset B^4$ whose double branched cover satisfies~$\Sigma_2(D) \cong B^4$.
\item  [(3)] When the fixed-point set is~$3$-dimensional,  it is an integer homology~$3$-sphere with a 3-ball removed.
The set of conjugacy classes of locally linear involutions with~$3$-dimensional fixed-point sets is in bijection with the set of homeomorphism classes of integer homology~$3$-spheres.
\item [(4)] When the fixed-point set is~$4$-dimensional, the involution is the identity.
\end{enumerate}
\end{theorem}

\begin{proof}
Observe that for any locally linear involution on~$B^4$, one can equivariantly glue the boundary to another~$B^4$ with a linear involution to obtain a locally linear involution on~$S^4$.
This is possible because any involution on the boundary~$S^3$ is linear; see Proposition \ref{prop:InvolutionsOnS3}. 

This produces a well defined map from the set of conjugacy classes of locally linear involutions on $B^4$ to the set of locally linear involutions on $S^4$.
When the fixed-point set~$F$ is $n$-dimensional,  since $\partial F$ is necessarily an $(n-1)$-sphere,  the fixed-point set of the capped off involution is obtained by capping off $F$ with an $n$-ball. 
This map is surjective as can be seen by removing an equivariant tubular neighborhood of a fixed-point,  and restricting the involution.

When the fixed-point set is~$0$ or~$1$-dimensional, Theorem \ref{thm:IntroTest} ensures that this action on~$S^4$ is conjugate to the standard linear action (assuming the existence of an equivariant tubular neighborhood in the~$1$-dimensional case).
The equivariant Schoenflies theorem (Theorem \ref{thm:EquivariantSchoenfliesIntro}) now implies that the original involution on~$B^4$ is conjugate to the standard linear action.

Next,  we consider the case where the fixed-point set is $3$-dimensional.
Since we already proved that every locally linear involution on $S^4$ arises by capping off a locally linear involution on $B^4$ with a copy of~$(B^4,\rho_{std})$, it only remains to establish the injectivity of this construction. 
For this,  given a locally linear involution $\rho \colon S^4 \to S^4$,  it is enough to show that for any pair of points~$x_0, x_1$ in the fixed-point set~$F$, and any choice of equivariant tubular neighborhoods~$\nu_0$ of~$x_0$ and~$\nu_1$ of~$x_1$, the exteriors~$(S^4 - \nu_0, \rho)$ and~$(S^4 - \nu_1, \rho)$ are equivariantly homeomorphic. 
In fact, we will observe that there is an equivariant homeomorphism of pairs~$(S^4, \nu_0) \cong (S^4, \nu_1)$. To see this, first find an equivariant tubular neighborhood~$\nu$ of~$F$ which contains both~$\nu_0$ and~$\nu_1$, equivariantly shrinking~$\nu_0$ and~$\nu_1$ if necessary. Then we can find an equivariant~$4$-ball~$B \subset \nu$ containing both~$\nu_0$ and~$\nu_1$. 
Apply the equivariant annulus theorem \ref{thm:EquivariantAnnulusThm}, slightly enlarging~$B$ by an equivariant collar (using Theorem~\ref{thm:equivariantcollaring}), to see that there is a equivariant homeomorphism of pairs~$(B,\nu_0) \cong (B,\nu_1)$ that is the identity on~$\partial B$.
Extend this homeomorphism by the identity over the rest of $B^4$ to obtain an equivariant homeomorphism of pairs~$(S^4, \nu_0, \rho) \cong (S^4, \nu_1, \rho)$ which is supported in~$B$.

Finally, we turn to the case where the fixed-point set is~$2$-dimensional. 
In this case, we find it more convenient to adapt the argument of Proposition~\ref{prop:2D} from $S^4$ to $B^4$.
For this, we need the following fact.
\begin{claim} \label{claim:D4quotient}
If~$\rho \colon B^4 \to B^4$ is a locally linear involution with a 2-dimensional fixed-point set, then~$B^4 / \rho$ is homeomorphic to~$B^4$. 
\end{claim}
\begin{proof}[Proof of Claim \ref{claim:D4quotient}]
As described above, one can equivariantly glue~$(B^4, \rho_{std})$ to~$(B^4, \rho)$ to obtain a locally linear involution~$\rho'$ on~$S^4$ with a~$2$-dimensional fixed-point set. The quotient~$S^4 / \rho'$ is homeomorphic to~$S^4$ by Lemma \ref{lem:SimplyConnectedOrbifold}. Then since~$B^4 /\rho_{std}$ is homeomorphic to~$B^4$, we have that~$B^4 / \rho \cong S^4/\rho' - B^4 / \rho_{std} \cong S^4 - B^4 \cong B^4$, where the last homeomorphism is a consequence of the Schoenflies theorem. 
\end{proof}
Now Theorem \ref{thm:D4involutions}(2) follows from the same argument as Proposition \ref{prop:2D}, noting that the mapping class group of~$B^4$ is also trivial (using the Alexander trick) and replacing Lemma \ref{lem:SimplyConnectedOrbifold} with Claim~\ref{claim:D4quotient}. 
\end{proof}

\appendix

\section{Locally linear involutions of~$S^3$}

The orthogonal group $O(4)$ acts on $S^3$ by diffeomorphisms,  and a diffeomorphism of~$S^3$ is \emph{orthogonal} if it arises in this way.
We note the following result on locally linear involutions of $S^3$ as it was mentioned in Remark~\ref{rem:InvolutionsS3} and will be used in the next section.

\begin{proposition}
\label{prop:InvolutionsOnS3}
Every locally linear involution on~$S^3$ is conjugate to an orthogonal involution.
\end{proposition}

\begin{proof}
When the fixed-point set is empty or $0$-dimensional,  this is work of Livesay~\cite{LivesayFixedPointFree,Livesay}.
When the fixed-point set is~$1$-dimensional, this is equivalent to the resolution of the Smith conjecture \cite[page 5]{MorganBass}. 
When the fixed-point set is $2$-dimensional, this follows from the Schoenflies theorem~\cite{BrownSchoenflies,MazurSchoenflies}: if two locally linear involutions~$\rho_0,\rho_1$ fix a~$2$-sphere,
then by the Schoenflies theorem, $\Fix(\rho_0)$ and $\Fix(\rho_1)$ bound balls on both sides, so $S^3/\rho_i \cong D^3$, and lifting any homeomorphism $S^3/\rho_0 \to S^3/\rho_1$ to the branched covers leads to the required conjugacy.
\end{proof}

Note that the conjugacy class of any involution in~$O(4)$ is uniquely determined by the dimension of the fixed-point set.

\section{An equivariant annulus theorem}

The following result answers Question~\ref{q:EAC} (which concerns an equivariant generalization of the annulus theorem) in the affirmative when the fixed-point set is~$3$-dimensional.

\begin{theorem}
\label{thm:EquivariantAnnulusThm}
Let~$\rho_{std}$ be a linear involution on~$S^4$,  let~$S^3_1,S^3_2 \subset S^4$ be~$3$-spheres that are fixed setwise by~$\rho_{std},$ and let~$A \subset S^4$ be the region cobounded by~$S^3_1$ and~$S^3_2.$
If the fixed-point set of~$\rho_{std}$ is a~$3$-dimensional fixed-point set, then~$(A, \rho_{std} \mid_{A})$ is equivariantly homeomorphic to~$(S^3 \times I, \rho_3 \times \id)$, for some linear involution~$\rho_3$ on~$S^3$.
\end{theorem}
\begin{proof}
Since~$A$ is homeomorphic to~$S^3 \times I$ by Quinn's annulus theorem~\cite{QuinnEndsIII},  we occasionally identify~$A$ with~$S^3 \times I$. 
We also note that the linearity of~$\rho_{std}$ ensures both that~$\Fix(\rho_{std}) \cong S^3$ and that~$\rho_{std}$ exchanges the two components of~$S^4 - \Fix(\rho_{std})$.
\begin{claim}
The fixed-point set of~$\rho_{std} \mid_{S^3 \times I}$ is homeomorphic to~$S^2 \times I$ and is a properly embedded separating submanifold of~$A = S^3 \times I$.
\end{claim}
\begin{proof}
The two boundary components of~$S^3 \times I$ are~$\rho_{std}$-invariant and therefore intersect~$\Fix(\rho_{std})$. 
We assert that~$\Fix(\rho_{std}) \cap (S^3 \times I)$ is a properly embedded submanifold of~$S^3 \times I=:A$.
Since~$A \subset S^4$ has codimension~$0$ and~$\Fix(\rho_{std}) \subset S^4$ is a submanifold, it suffices to verify that~$\partial A~$ and 
$\Fix(\rho_{std})$ intersect transversally.
Theorem \ref{thm:equivariantcollaring} ensures that~$\partial A=S^3 \times \{0,1\}$ admits an equivariant collar neigbhorhood of the form~$I \times \partial A$.
This implies that~$(I \times \partial A) \cap \operatorname{Fix}(\rho_{std})=\operatorname{Fix}(\rho_{std}|_{\partial A}) \times I$ and the asserted transversality follows.

Smith theory (see e.g.  Floyd~\cite[Theorem 4.4]{Floyd}) implies
 that~$\operatorname{Fix}(\rho_{std}|_{S^3 \times \{i\}})$ is homeomorphic to~$S^0,S^1$ or~$S^2$ for~$i=0,1$.
Since~$\partial A$ and~$\operatorname{Fix}(\rho_{std}|_A)$ are $3$-manifolds that intersect transversally in a~$4$-manifold,  their intersection is a surface.
We deduce from these facts that~$\operatorname{Fix}(\rho_{std}|_{\partial A})\cong S^2 \sqcup S^2$.

The Schoenflies theorem implies that the components of~$S^2 \sqcup S^2\cong \operatorname{Fix}(\rho_{std}|_{\partial A}) \subset \operatorname{Fix}(\rho_{std})\cong~S^3$ bound disjoint embedded balls~\cite{BrownSchoenflies,MazurSchoenflies}.
The exterior of these two $3$-dimensional balls agrees with~$\Fix(\rho_{std}) \cap A$.
The annulus theorem then implies that~$\Fix(\rho_{std}) \cap A$ is homeomorphic to~$S^2 \times I$; as indicated e.g. in~\cite[below Theorem 4.1]{FriedlNagelOrsonPowell},
in this dimension,  the annulus theorem follows from work of Rad\'o~\cite{Rado} and Moise~\cite{MoiseHauptvermutung,MoiseGeometricTopology}.
Since~$\Fix(\rho_{std}) \cap A$ is properly embedded in~$A \cong (S^3\times I)$, it follows that it is separating, as claimed.
\end{proof}

We now assert that the two components of~$A - \nu$ are homotopy balls with boundary $S^3$. 
Since~$\rho_{std}$ is a smooth action,  the fixed-point set~$S^2 \times I \cong  \operatorname{Fix}(\rho_{std}|_A):=F_A \subset A$ admits an equivariant tubular neighborhood~$\nu$.
Write~$A  - F_A=A_1 \sqcup A_2$, thicken~$A_1,A_2$ to~$U_1,U_2$ so that~$A=U_1 \cup_{\nu} U_2$.
The fact that the~$A_i$ are simply-connected follows readily from van Kampen and,  since~$A \cong S^3 \times I$ and~$\nu \cong (S^2 \times I) \times I$,  the Mayer-Vietoris sequence takes the form
$$
\xymatrix{
0 \ar[r]& H_3(A_1) \oplus H_3(A_2) \ar[r]& 
H_3(A) \ar[r]^-{\partial}& H_2(\nu) 
\ar[r]& H_2(A_1) \oplus H_2(A_2)
\ar[r]& 0. \\
&&H_3(\partial_+A)\ar[r]^-{\partial,\cong} \ar[u]^{\cong}&H_2(\partial_+ F_A)\ar[u]^{\cong}&
}
$$
To see that the bottom horizontal map is an isomorphism,  one can for example apply the Schoenflies theorem to the sphere~$\partial_+F_A \cong S^2$ in~$\partial_+A\cong S^3$.
Finally, since we proved that the restriction of the action to the components of $\partial A \cong S^3 \sqcup S^3$ fixes an $S^2$,  the boundary of the components of~$A- \nu$ are homeomorphic to~$B^3 \cup (S^2 \times I) \cup B^3 \cong S^3$.
This establishes the assertion.

Since the components of~$A - \nu$ are homotopy balls with~$S^3$ boundary  by work of Freedman,  they are homeomorphic to~$B^4$~\cite{Freedman}.
At this stage, we have proved that~$\rho_{std} \mid_A$ has fixed-point set~$S^2 \times I$, and~$A -\nu \cong B^4 \sqcup B^4$ with the components exchanged by the symmetry. 

It remains to conclude that~$(A, \rho_{std} \mid_A)$ is equivariantly homeomorphic~$(S^3 \times I, \rho_3 \times \id)$, where~$\rho_3$ is the linear involution on~$S^3$ given by reflection across an~$S^2$. 
The orbit set~$(S^3 \times I)/(\rho_3 \times \id)$ is homeomorphic to~$B^4$.
To see that~$A/\rho$ is also homeomorphic to $B^4$, observe that~$A- \Fix(\rho)$ consists of two~$4$-balls (that are permuted by the involution) and that~$\partial (A/\rho)$ is homeomorphic to a $3$-sphere.
To see the latter, note that~$\partial (A/\rho)$ is made up of~$(\partial A)/\rho \cong B^3 \sqcup B^3$ (for this homeomorphism we use Proposition~\ref{prop:InvolutionsOnS3}) and~$\Fix(\rho)/\rho =S^2\times I$ that combine into a~$3$-sphere.
Next, we construct a homeomorphism $A/\rho \to (S^3 \times I)/(\rho_3 \times I).$
As described above, the boundaries of these two~$4$-balls decompose as~$B^3 \cup (S^2 \times I) \cup B^3$.
Map the~$S^2 \times I$ subspaces homeomorphically to each other.
Then extend this homeomorphism over the~$3$-balls, thus yielding a homeomorphism~$\partial (A/\rho) \to \partial (S^3 \times I)/(\rho_3 \times \id_I))$.
Further extend over the~$4$-ball to a homeomorphism~$A/\rho \to (S^3 \times I)/\rho_3 \times \id$.
Lifting to the branched covers over~$S^2 \times I$, this leads to the required equivariant homeomorphism~$(A, \rho_{std} \mid_A)\cong(S^3 \times I, \rho_3 \times \id)$.
\end{proof}

\def\MR#1{}
\bibliography{bibliosimpleCP2}
\end{document}